\documentclass[final,a4paper]{siamltex}
\usepackage{amsmath,amssymb,amsfonts}
\usepackage{color}
\usepackage{verbatim}
\usepackage{eepic,epic,epsfig,epstopdf}
\usepackage[]{graphics}
\usepackage{graphicx,inputenc}
\usepackage{subfigure}
\usepackage{changepage}
\usepackage{amsmath}
\usepackage{threeparttable}
\usepackage{extarrows}
\usepackage{url,color}
\usepackage{graphicx,subfig}
\graphicspath{{./pics/}}
\usepackage{bm,color,url}
\usepackage{algorithm,algorithmic}
\usepackage{geometry}
\usepackage{setspace}
\newtheorem{remark}{Remark}[section]

\def\calA{{\mathcal{A}}}
\def\calI{{\mathcal{I}}}

\title{ Robust Decoding from  1-Bit Compressive Sampling with Least Squares }
\author{Jian Huang\thanks{Department of Applied Mathematics, The Hong Kong Polytechnic University,  Hong Kong 999077, P.R. China.(j.huang@polyu.edu.hk)}\quad\and
Yuling Jiao\thanks{School of Statistics and Mathematics, Zhongnan University of Economics and Law, Wuhan 430063, P.R. China. (yulingjiaomath@whu.edu.cn)}\quad\and
 Xiliang Lu\thanks{School of Mathematics and Statistics, and
Hubei Key Laboratory of Computational Science, Wuhan University, Wuhan 430072, P.R. China. (xllv.math@whu.edu.cn)}\quad\and
 Liping Zhu\thanks{Institute of Statistics and Big Data and Center for Applied Statistics, Renmin University of China, 	Beijing 100872, P. R. China. (zhu.liping@ruc.edu.cn)}
}

\begin{document}
\maketitle

\begin{abstract}
In  1-bit compressive sensing (1-bit CS) where  target signal  is coded into  a  binary  measurement, one goal is to   recover the signal   from noisy and   quantized   samples. Mathematically, the 1-bit CS model reads: $y = \eta \odot\textrm{sign} (\Psi x^* + \epsilon)$, where $x^{*}\in \mathcal{R}^{n},  y\in \mathcal{R}^{m}$,  $\Psi \in \mathcal{R}^{m\times n}$, and $\epsilon$ is the random  error before quantization and  $\eta\in \mathcal{R}^{n}$ is a random vector modeling the sign flips. Due to the presence of  nonlinearity, noise  and sign flips, it is quite challenging to decode from the 1-bit CS.  In this paper,
we consider least squares approach under  the over-determined and  under-determined settings.  For  $m>n$, we show that, up to a constant $c$, with high probability, the least squares solution $x_{\textrm{ls}}$ approximates  $ x^*$ with precision $\delta$ as long  as $m \geq\widetilde{\mathcal{O}}(\frac{n}{\delta^2})$.
For $m< n$,
 we prove that, up to a constant $c$, with high probability, the $\ell_1$-regularized least-squares solution  $x_{\ell_1}$ lies in the ball with center   $x^*$  and  radius $\delta$  provided that $m \geq \mathcal{O}( \frac{s\log n}{\delta^2})$ and  $\|x^*\|_0 := s < m$.
We introduce  a  Newton type method, the so-called  primal and dual active set (PDAS) algorithm,  to solve  the nonsmooth optimization problem. The PDAS possesses the property of   one-step convergence. It only requires to solve
 a small least
squares problem on the active set. Therefore,   the PDAS is   extremely efficient for recovering  sparse signals
  through continuation. We propose a novel regularization parameter selection  rule which  does not introduce  any extra computational overhead.
Extensive numerical experiments are presented to illustrate the robustness  of our proposed
model and the efficiency  of our  algorithm. 
\end{abstract}

\section{Introduction}\label{sec:intro}
Compressive sensing (CS)
is an important approach to  acquiring low dimension
signals from noisy under-determined measurements \cite{CandesRombergTao:2006,Donoho:2006,FazelCandesRcht:2008,FoucartRauhut:2013}.
For storage and transmission,
the infinite-precision measurements are often quantized, \cite{BoufounosBaraniuk:2008} considered recovering the signals  from the 1-bit compressive sensing  (1-bit CS)  where  measurements
are  coded into a single bit, i.e., their signs. The 1-bit CS is superior to the  CS
in terms of inexpensive hardware implementation and storage.
However, it is much more challenging to decode from nonlinear, noisy and sign-flipped  1-bit measurements.
 \subsection{Previous work}
 Since the seminal  work of  \cite{BoufounosBaraniuk:2008},
much effort has been devoted to studying the theoretical and computational challenges  of the 1-bit CS.
Sample complexity was analyzed for    support and  vector recovery  with and without noise \cite{GopiJain:2013,JacquesDegraux:2013,PlanVershynincpam:2013,BauptBaraniuk:2011,JacquesLaska:2013, GuptaNowakRech:2010,BauptBaraniuk:2011,PlanVershynin:2013,ZhangYiJin:2014}.
Existing works indicate  that,  $m > \mathcal{O}(s\log n)$ is adequate for both  support and  vector recovery. The sample size required here
has the same order as that required
in the standard CS setting.
These results have also been refined by   adaptive sampling  \cite{GuptaNowakRech:2010,DaiShenXuZhang:2016,BaraniukFoucartNeedellPlan:2017}. Extensions include   recovering the norm of the target \cite{KnudsonSaabWard:2016,BaraniukFoucartNeedellPlan:2016} and   non-Gaussian measurement settings \cite{AiPlanVershynin:2014}.
 Many first order methods   \cite{BoufounosBaraniuk:2008, LaskaWenYinBaraniuk:2011, YanYangOsher:2012,DaiShenXuZhang:2016} and greedy methods \cite{LiuGongXu:2016,Boufounos:2009,JacquesLaska:2013} are developed  to minimize the sparsity promoting nonconvex objected function arising  from either  the unit sphere constraint or the  nonconvex regularizers.
To address the nonconvex optimization  problem,  convex  relaxation models are also proposed \cite{ZhangYiJin:2014,PlanVershynin:2013,PlanVershynincpam:2013,
ZymnisBoydCandes:2010,PlanVershynin:2017}, which often  yield accurate solutions efficiently with polynomial-time solvers. See, for example, \cite{Nesterov:2013}.

\subsection{1-bit CS setting}\label{setting}
 In this paper we consider the following  1-bit CS model
 \begin{equation}\label{setup}
 y = \eta \odot\textrm{sign} (\Psi x^* + \epsilon),
  \end{equation}
where $y\in \mathcal{R}^{m}$ is the 1-bit measurement, $x^{*}\in \mathcal{R}^{n}$ is an unknown signal, $\Psi = [\psi_1^t;...;\psi_m^t] \in \mathcal{R}^{m\times n}$ is a random  matrix, $\eta\in \mathcal{R}^{m}$ is a random vector modeling the sign flips of $y$, and $\epsilon \in \mathcal{R}^{n}$ is a random vector with  independent and identically distributed (iid) entries modeling errors before quantization. Throughout $\textrm{sign}(\cdot)$ operates componentwise with $\textrm{sign}(z) =1$ if $z \geq 0$ and $\textrm{sign}(z) =-1$ otherwise,  and $\odot $ is the pointwise Hardmard product. Following \cite{PlanVershynincpam:2013} we assume that the rows of $\Psi$ are  iid random vectors sampled from the multivariate normal distribution $\mathcal{N}(\textbf{0},\Sigma)$ with  an unknown
covariance matrix $\Sigma $, $\epsilon$ is
distributed as $\mathcal{N}(\textbf{0},\sigma^2\textbf{I}_m)$ with an
unknown noise level $\sigma$, and
$\eta\in \mathcal{R}^{m}$ has independent coordinates $\eta_is$ satisfying
 $\mathbb{P}[\eta_i = 1] = 1- \mathbb{P}[\eta_i = -1] = q \neq \frac{1}{2}$.  We assume  $\eta_i, \epsilon_i$ and $\psi_i$ are mutually independent.
Because $\sigma$ is known model \eqref{setup} is  invariant in the sense that  $\forall \alpha > 0$, $y = \eta \odot\textrm{sign} (\Psi x^* + \epsilon) =  \eta \odot\textrm{sign} (\alpha\Psi x^* + \alpha \epsilon)$. This indicates that the best one can hope for  is to recover $x^*$ up to a scale factor.  Without loss of generality
we assume $\| x^*\|_{\Sigma} = 1$.

\subsection{Contributions}

We study the 1-bit CS problem in both the overdetermined setting with $m > n$  and  the underdetermined setting with $m<n$. In the former setting we allow for  dense $x^*$, while in the latter, we assume that $x^*$ is sparse in the sense that  $\|x^*\|_0 = s< m.$
The basic message is that we can recover $x^*$
 with the ordinary  least squares or the $\ell_1$ regularized  least squares.

(1) When $m>n$, we propose to use  the least squares solution
\begin{equation*}
 x_{\textrm{ls}} \in \arg \min\frac{1}{m}\sum_{i=1}^{m} (y_i - \psi_i^t x)^2
\end{equation*}
  to approximate $x^*$. We show that, with high probability,
$x_{\textrm{ls}}$ estimates  $x^*$ accurately up to a positive scale factor $c$ defined by \eqref{constc} in the sense that,
$\forall \delta \in (0,1)$, $\|x_{\textrm{ls}}/c- x^*\|\leq \delta$ if  $m\geq \widetilde{\mathcal{O}}(\frac{n}{\delta^2})$.
We make the following observation:
\begin{adjustwidth}{1cm}{1cm}
   \textit{Up to a constant $c$, the underlying target $x^*$ can be decoded  from 1-bit measurements  with the ordinary  least squares,
   as long as the probability of sign flips  probability is not equal to $1/2$.
   }
 \end{adjustwidth}
 (2) When $m<n$ and the target signal $x^*$ is sparse,
we consider the $\ell_1$-regularized least squares solution
 \begin{equation}\label{subreg}
 x_{\ell_1} \in \arg \min  \frac{1}{2m}\|y - \Psi x\|_2^2 + \lambda \|x\|_1.
  \end{equation}
The sparsity assumption  is widely used   in modern signal processing \cite{FoucartRauhut:2013,Mallat:2008}.
  We show that, with high probability  the error  $\|x_{\ell_1}/c - x^*\|$ can be bounded by a prefixed accuracy
   $\delta \in(0,1)$   if  $m\geq \mathcal{O}(\frac{s\log n}{\delta^2})$, which is the same as the  order
for  the standard CS methods to work.
   Furthermore,  the support of $x^*$ can be exactly recovered if the minimum signal magnitude of $x^*$ is larger than  $\mathcal{O}(\sqrt{s\log n/m}).$
   When the target signal is sparse,  we  obtain the following   conclusion:
\begin{adjustwidth}{1cm}{1cm}
   \textit{Up to a constant $c$, the  sparse signal  $x^*$ can also be decoded from 1-bit measurements with the  $\ell_1$-regularized least squares, as long as the probability of sign flips  probability is not equal to $1/2$.
   }
 \end{adjustwidth}
 (3) We introduce   a fast and accurate   Newton method, the so-called   primal dual active set method (PDAS), to solve the $\ell_1$-regularized minimization \eqref{subreg}.  The  PDAS  possesses the property of    one-step convergence. The PDAS solves a small least
squares problem on the active set,  is thus extremely efficient if  combined with continuation.
 We propose a novel regularization parameter selection rule,   which is  incorporated with continuation  procedure without additional cost.
{The code is available  at \url{http://faculty.zuel.edu.cn/tjyjxxy/jyl/list.htm}.}

 \begin{adjustwidth}{1cm}{1cm}
   \textit{The optimal solution $x_{\ell_1}$ can be computed    efficiently and  accurately  with  the PDAS, a   Newton type  method which converges after one iteration,  even if the objective function  \eqref{subreg} is nonsmooth. Continuation on $\lambda$ globalizes the  PDAS. The regularization parameter can be  automatically determined  without additional computational cost.}
 \end{adjustwidth}

\subsection{Notation and organization}
Throughout we denote by $\Psi_i\in \mathcal{R}^{m\times1}, i = 1,..., m,$ and $\psi_j \in \mathcal{R}^{n\times 1},j=1,...,n$ the $i$th column and $j$th row  of $\Psi$, respectively. We denote a vector of  $0$  by $\textbf{0}$, whose length may vary in different places. We use   $[n]$  to denote the set $\{1,...,n\}$, and  $\textbf{I}_n$ to denote the identity matrix of size $n\times n$.
 For  $ A,B\subseteq [n]$ with length $|A|,|B|$, $x_{A}=(x_{i}, i\in A)\in \mathbb{R}^{|A|}$,  $\Psi_{A}=(\Psi_{i}, i\in A)\in \mathbb{R}^{m\times|A|}$ and $\Psi_{AB}\in \mathbb{R}^{|A|\times
|B|}$  denotes a submatrix of $\Psi$ whose rows and columns
are listed in $A$ and $B$, respectively.
We  use $(\psi_i)_j $ to denote
the $j$th entry of the  vector $\psi_i$, and  $|x|_{\textrm{min}}$ to denote the minimum absolute value of $x$.
We use $\mathcal{N}(\textbf{0},\Sigma)$ to denote the multivariate normal distribution, with $\Sigma$  symmetric and positive definite.  Let $\gamma_{\textrm{max}}(\Sigma)$ and  $\gamma_{\textrm{min}}(\Sigma)$  be the largest  and the smallest eigenvalues of $\Sigma$, respectively,  and $\kappa(\Sigma)$ be the condition number $\gamma_{\textrm{max}}(\Sigma)/\gamma_{\textrm{min}}(\Sigma)$ of $\Sigma$. We use $\|x\|_{\Sigma}$ to denote the elliptic norm of $x$ with respect to $\Sigma$, i.e., $\|x\|_{\Sigma} = (x^{t}\Sigma x)^{\frac{1}{2}}.$ Let $\|x\|_p = (\sum_{i=1}^{n}|x_{i}|^p)^{1/p}, p\in [1,\infty],$  be
the $\ell_p$-norm of  $x$.
We denote the  number of nonzero elements of $x$ by $\|x\|_{0}$ and let $s = \|x^*\|_{0}$. The symbols $\|\Psi\|$ and $\|\Psi\|_{\infty}$ stands for the operator norm of $\Psi$ induced by $\ell_2$ norm and the maximum  pointwise absolute value of $\Psi$, respectively. We use $\mathbb{E}[\cdot]$, $\mathbb{E}[\cdot|\cdot]$, $\mathbb{P}[\cdot]$ to denote the expectation, the conditional expectation  and the probability on a given probability space $(\Omega, \mathfrak{F},\mathbb{P})$. We use $C_1$ and $C_2$ to denote generic  constants which may  vary from place to place.
By  $\mathcal{O}(\cdot)$ and $\widetilde{\mathcal{O}}(\cdot)$, we  ignore some positive numerical  constant and $\sqrt{\log n}$, respectively.

The rest of the paper is organized as follows. In Section 2  we explain
why the  least squares works in the 1-bit CS  when $m>n$, and obtain   a nonasymptotic error bound for $\| x_{\textrm{ls}}/c-x^* \|.$
In Section 3  we  consider the sparse decoding
when $m<n$ and prove a minimax bound on   $\|x_{\ell_1}/c -x^* \|$.
 In Section 4 we introduce the  PDAS algorithm to solve  \eqref{subreg}. We propose a new regularization parameter selection rule combined with the  continuation  procedure.
In Section 5 we conduct simulation studies   and compare  with  existing 1-bit CS  methods.
We conclude with some  remarks in Section 6.

\section{Least squares when $m>n$}
In this section, we first explain
why the least squares  works  in the over-determined  1-bit CS model \eqref{setup} with  $m>n$.
We then prove  a nonasymptotic error bound on $\|x_{\textrm{ls}}/c - x^* \|.$
The following lemma inspired by
 \cite{Brillinger:2012} is our starting point.
\begin{lemma}
Let $\tilde{y} = \tilde{\eta}  \textrm{sign}(\tilde{\psi}^t x^* + \tilde{\epsilon})$ be  the  1-bit model \eqref{setup} at the population level.  $\mathbb{P}[\tilde{\eta} = 1] = q \neq \frac{1}{2}$, $\tilde{\psi}\sim \mathcal{N}(\textbf{0},\Sigma)$, $\tilde{\epsilon}\sim \mathcal{N}(0,\sigma^2).$
It follows that,
\begin{equation}\label{parallel}
\Sigma^{-1} \mathbb{E}[\tilde{y} \tilde{\psi}]/c = x^*,
\end{equation}
\end{lemma}
where
\begin{equation}\label{constc}
  c = (2q-1) \sqrt{\frac{2}{\pi (\sigma^2 + 1)}}.
  \end{equation}
\begin{proof}
The proof  is given in Appendix \ref{app:parallel}.
\end{proof}

Lemma \ref{parallel}  shows that,  the target $x^*$  is proportional to  $\Sigma^{-1} \mathbb{E}[\tilde{y} \tilde{\psi}]$.
Note that
\begin{align}
\mathbb{E}[\Psi^t\Psi/m]&= \mathbb{E}[\sum_{i=1}^m \psi_i \psi_i^t]/m   = \Sigma,\label{ub1} \quad and \quad\\
 \mathbb{E}[\Psi^t y/m ]&=  \mathbb{E}[\sum_{i = 1}^{m} \psi_i y_i ]/m = \mathbb{E}[\tilde{y} \tilde{\psi}]\label{ub2}.
  \end{align}
  As long as $\Psi^t \Psi/m$ is invertible,  it is reasonable to expect    that  $$x_{\textrm{ls}} = (\Psi^t \Psi/m)^{-1} (\Psi^t y/m) = (\Psi^t \Psi)^{-1} (\Psi^t y)$$
can approximate  $x^*$ well up to a constant $c$ even if $y$ consists of 
sign flips.

\begin{theorem}\label{errls}
Consider the ordinary least squares:
\begin{equation}\label{ls}
x_{\textrm{ls}}\in \arg\min_{x} \frac{1}{m}\|\Psi x - y\|_2^2.
\end{equation}
 If $m \geq 16C^2_2 n$, then
with probability at least $1- 4\exp{(-C_1 C_2^2 n)} -\frac{2}{n^3}$,
\begin{equation}\label{lsnorm}
\|x_{\textrm{ls}}/c- x^*\|_2 \leq   \sqrt{\frac{n}{m}}(4C_2 \sqrt{\kappa(\Sigma)\gamma_{\max}(\Sigma)}  + \frac{6(\sigma+1)}{\sqrt{C_1}|2q-1|}\sqrt{\log n}),
\end{equation}
where  $C_1$ and $C_2$ are some generic  constants not depending on $m$ or $n$.
\end{theorem}

 \begin{proof}
The proof  is given in Appendix \ref{app:errls}.
\end{proof}

\begin{remark}
Theorem \eqref{errls} shows that, $\forall \delta \in (0,1)$ if  $m \geq \widetilde{\mathcal{O}}(\frac{n}{\delta^2})$,  up to a constant, the simple least squares solution can  approximate  $x^*$  with  error of order $\delta$  even if  $y$ contains very large quantization error  and sign flips with probability unequal to  $1/2$.
\end{remark}

\begin{remark}
To the best of our knowledge, this is the first nonasymptotic error bound for the  1-bit CS in the overdetermined setting.
Comparing with the estimation error of the standard least squares in the complete
data model $y = \Psi x^* + \epsilon$,
the error bound in Theorem \ref{errls} is optimal up to a logarithm factor $\sqrt{\log n}$, which is due to the loss of information with the  1-bit quantization.
\end{remark}

\section{Sparse decoding with $\ell_1$-regularized least squares}

\subsection{Nonasymptotic error bound}
Since  images and signals are often  sparsely represented  under certain transforms \cite{Mallat:2008,DongShen:2010}, it suffices for the
standard  CS   to  recover the  sparse signal $x^*\in \mathcal{R}^{n}$ with  $m = \mathcal{O}(s \log n)$ measurements for  $s = \|x^*\|_0$ .  In this section we show that  in the  1-bit CS setting, similar results  can be derived through the  $\ell_1$-regularized least squares \eqref{subreg}.
  Model \eqref{subreg} has been extensively studied when $y$ is continuous   \cite{Tibshirani:1996,Chen:1998,CandesRombergTao:2006,Donoho:2006}. Here we use model \eqref{subreg} to recover $x^*$ from quantized $y$, which is rarely studied in the literature.

Next we show that, up to the constant $c$, $x_{\ell_{1}}$ is a good estimate of $x^*$ when $m = \mathcal{O}(s \log n)$ even if the signal  is highly noisy and  corrupted by sign flips in the 1-bit CS setting.

 \begin{theorem}\label{errsub}
 Assume
$ n> m \geq  \max\{\frac{4C_1}{C_2^2}\log n, \frac{64(4\kappa(\Sigma)+1)^2}{C_1}s\log\frac{en}{s}\}$,  $s \leq \exp^{(1-\frac{C_1}{2})} n$.  Set $\lambda = \frac{4(1+C_3|c|)}{\sqrt{C_1}}\sqrt{\frac{\log n}{m}} $. Then
with probability at least $1-2/n^3  -6/n^2$, we have,
\begin{equation}\label{lssub}
 \| x_{\ell_{1}}/c - x^*\|_2\leq \frac{816(4\kappa(\Sigma)+1)^2}{\gamma_{min}(\Sigma)} \frac{\sigma+1+C_3 |q-1/2|}{\sqrt{C_1}|q-1/2|}\sqrt{\frac{s\log n}{m}}.
\end{equation}
\end{theorem}

 \begin{proof}
The proof  is given in Appendix \ref{app:errsub}.
\end{proof}

\begin{remark}
Theorem \ref{errsub} shows that, $\forall \delta \in (0,1)$, if  $m \geq \mathcal{O}(\frac{s \log n }{ \delta^2})$,  up to a constant $c$, the $\ell_1$-regularized least squares solution can  approximate  $x^*$  with precision  $\delta$.
\end{remark}

\begin{remark}
The error bound in Theorem \ref{errsub} achieves  the minimax optimal order  $\mathcal{O}(\sqrt{\frac{s \log n}{m}})$ in the sense that  it is the optimal order that can be attained even if  the signal  is measured precisely without 1-bit quantization \cite{Negahban:2012}. From   Theorem \ref{errsub} if the minimum nonzero  magnitude of $x^*$ is large enough, i.e., $|x^*|_{\min} \geq \mathcal{O}(\sqrt{\frac{s\log n}{m}})$, the support of $x_{\ell_{1}}$ coincides with that of $x^*$. 
\end{remark}

\subsection{Comparison with related  works}
Assuming   $\|x^*\|_2 = 1$ and $\sigma = 0$ and $q = 1$, \cite{BoufounosBaraniuk:2008} proposed to decode $x^*$ with \begin{equation*}
\min_{x\in \mathcal{R}^n} \|x\|_1\quad s.t.\quad y\odot \Psi x \geq  0, \quad \|x\|_2 = 1.
 \end{equation*}
 A first order algorithm was devised  to solve the following  Lagrangian version  \cite{LaskaWenYinBaraniuk:2011}, i.e.,
   \begin{equation*}
\min_{x\in \mathcal{R}^n}  \|\max\{\textbf{0},-y\odot \Psi x\}\|_2^2 + \lambda\|x\|_1\quad s.t. \quad \|x\|_2 = 1.
 \end{equation*}
 In the presence of   noise, \cite{JacquesLaska:2013} introduced
   \begin{equation}\label{orrobust}
\min_{x\in \mathcal{R}^n}  \mathcal{L}(\max\{\textbf{0},-y\odot \Psi x\}) \quad s.t. \quad \|x\|_0 \leq s, \quad \|x\|_2 = 1,
 \end{equation}
 where  $\mathcal{L}(\cdot) = \|\cdot\|_1$ or $\|\cdot\|_2^2$.
 They used  a projected sub-gradient method, the so-called  binary iterative hard thresholding (BITH), to solve  \eqref{orrobust}.
Assuming that there are sign flips in the noiseless model with  $\sigma = 0$, \cite{DaiShenXuZhang:2016} considered
 \begin{equation}\label{prox}
\min_{x\in \mathcal{R}^n}  \lambda \|\max\{\textbf{0},\nu \textbf{1}-y\odot \Psi x\}\|_{0} + \frac{\beta}{2} \|x\|_2^2 \quad s.t. \quad \|x\|_0 \leq s,
 \end{equation}
 where
 $\nu >  0,  \beta > 0 $ are tuning parameters.
 An  adaptive outlier pursuit (AOP)  generalization of \eqref{orrobust} was proposed in \cite{YanYangOsher:2012} to recover $x^*$ and  simultaneously  detect the entries with sign flips by
   \begin{equation*}
\min_{x\in \mathcal{R}^n, \Lambda \in \mathcal{R}^{m}} \mathcal{L}(\max\{\textbf{0},-\Lambda \odot y\odot \Psi x\}) \quad s.t. \quad \Lambda_i \in \{0,1\},\quad \|\textbf{1} - \Lambda\|_1 \leq N, \quad \|x\|_0 \leq s,  \quad \|x\|_2 = 1,
 \end{equation*}
 where $N$ is the  number of sign flips. Alternating minimization on $x$ and $\Lambda$ are adopted to solve the optimization problem.  \cite{HuangShiYan:2015} considered the pinball loss to deal with both the noise and the sign flips,  which reads
  \begin{equation*}
 \min_{x\in \mathcal{R}^n}  {\mathcal{L}}_{\tau}( \nu \textbf{1}  - y \odot \Psi x\}) \quad s.t. \quad \|x\|_0 \leq s  \quad \|x\|_2 = 1,
 \end{equation*}
 where $\mathcal{L}_{\tau}(t) = t \textbf{1}_{t\geq 0} - \tau t \textbf{1}_{t<0}$. Similar to the  BITH,  the pinball iterative hard thresholding   is  utilized.
  With the binary stable embedding, \cite{JacquesLaska:2013} and  \cite{DaiShenXuZhang:2016}  proved that with high probability, the sample complexity  of \eqref{orrobust} and \eqref{prox} to guarantee
estimation error smaller than $\delta$   is  $m \geq \mathcal{O}(\frac{s}{\delta^2} \log \frac{n}{s})$,  which echoes   Theorem \ref{errsub}.
 However,  there are no theoretical  results for  other models  mentioned  above.
 All the aforementioned models  and algorithms are  the state-of-the-art works  in the 1-bit CS.  However, all the methods mentioned above  are nonconvex. It is thus  hard to justify whether the  corresponding algorithms are loyal to their models.

 Another line of research  concerns convexification.
The pioneering  work   is \cite{PlanVershynincpam:2013}, where they considered the noiseless case without sign flips and   proposed  the following linear programming
method
 \begin{equation*}
 x_{lp}\in \arg\min_{x\in \mathcal{R}^n} \|x\|_1 \quad s.t. \quad  y \odot\Psi x\geq 0  \quad \|\Psi x\|_1  =  m.
 \end{equation*}
 As shown in \cite{PlanVershynincpam:2013},  the estimation error is   $\|\frac{x_{lp}}{\|x_{lp}\|} - x^*\|\leq \mathcal{{O}}(({s \log^2 \frac{n}{s}})^{\frac{1}{5}})$.
 The above result is improved to $ \|\frac{x_{cv}}{\|x_{cv}\|} - x^*\|\leq \mathcal{{O}}(({s \log \frac{n}{s}})^{\frac{1}{4}})$ in \cite{PlanVershynin:2013}, where both the noise and the sign flips are allowed,  through  considering the convex problem
  \begin{equation}\label{convex1}
 x_{cv} \in \arg\min_{x\in \mathcal{R}^n} -\langle y, \Psi x\rangle/m \quad s.t. \quad \|x\|_1 \leq s,  \quad \|x\|_2 \leq 1.
 \end{equation}
 Comparing  with   our result in Theorem \ref{errsub}, the results derived  in \cite{PlanVershynincpam:2013} and \cite{PlanVershynin:2013} are suboptimal.

 In the noiseless case and assuming $\Sigma = \textbf{I}_n$,  \cite{ZhangYiJin:2014} considered the Lagrangian version of  \eqref{convex1}
  \begin{equation}\label{convex2}
 \min_{x\in \mathcal{R}^n} -\langle y, \Psi x\rangle/m  +  \lambda \|x\|_1 \quad s.t.  \quad \|x\|_2 \leq 1.
 \end{equation}
 In this special case, the estimation error  derived in \cite{ZhangYiJin:2014} improved that of   \cite{PlanVershynin:2013} and matched our  results in  Theorem \ref{errsub}. However, \cite{ZhangYiJin:2014} did not  discuss the scenario of    $\Sigma \neq \textbf{I}_n.$
  Recently \cite{PlanVershynin:2017,Vershynin:2015},  proposed a simple projected linear estimator $P_{K}(\Psi^t y/m)$, where $K = \{x\big |\|x\|_1 \leq s, \|x\|_2 \leq 1\}$,  to estimate the low-dimensional structure target belonging to  $K$  in high dimensions from noisy and possibly nonlinear observations. They derived the same  order of estimation error as that in Theorem \ref{errsub}.

 \cite{ZymnisBoydCandes:2010} proposed an
 $\ell_1$ regularized maximum likelihood estimate,  and \cite{HuangShiYan:2015} introduced a convex model through  replacing the linear loss in \eqref{convex2} with the  pinball loss. However, neither   studied sample complexity or estimation error.


\section{Primal dual active set algorithm}

Existing algorithms  for \eqref{subreg}  are mainly first order methods \cite{csnmr,BJMG,pfbsr}.
In this section we use  primal dual active set method \cite{FanJiaoLu:2014,JiaoJinLu:2015}, which is  a generalized Newton type method, \cite{ItoKunisch:2008,QiSun:1993}  to solve \eqref{subreg}.
An important advantage of the  PDAS is that it converges after one-step iteration if the initial value is good enough.
We globalize it with  continuation  on regularization parameter. We also
propose a  novel  regularization parameter selection rule  which is  incorporated along the continuation  procedure without any additional computational burden.

\subsection{PDAS}
In this section we use $x$ to denote $x_{\ell_{1}}$ for simplicity.
We begin with
the  following  results \cite{Combettes:2005}.
Let $x$ be the minimizer of \eqref{subreg}, then $x$ satisfies
\begin{equation}\label{kkteq}
\left\{
                             \begin{array}{ll}
                               d = \Psi^t(y-\Psi x)/m, \\
                               x = \mathcal{S}_{\lambda }(x +  d).
                             \end{array}
                           \right.
\end{equation}
Conversely, if $x$ and $d$ satisfy \eqref{kkteq}, then $x$ is a global minimizer of \eqref{subreg},
where $\mathcal{S}_{\lambda }(z)$ is the pointwise  soft-thresholding operator defined by
 $\mathcal{S}_{\lambda }(z_i) = \arg \min_{t} \frac{1}{2}(t-z_i)^2+ \lambda|t|_{1} =  sign(z_i)\max\{{|z_i|-\lambda,0}\}
.$


Let $Z =\left(
         \begin{array}{c}
         x \\
          d  \\
         \end{array}
       \right) \ \mbox{ and } \
$
$
F(Z)=
\left(
\begin{array}{c}
 F_{1}(Z)\\
F_{2}(Z)
\end{array}
 \right)
:\mathcal{R}^{n} \times \mathcal{R}^{n} \rightarrow \mathcal{R}^{2n},
$
where
$
  F_{1}(Z)= x  - \mathcal{S}_{\lambda}(x + d) \text{ and }
  F_{2}(Z)= \Psi^t \Psi x + m d -\Psi^t y.
$
By \eqref{kkteq}, finding  the  minimizer $x$ of \eqref{subreg} is equivalent   to finding the  root of $F(Z)$.
We use the following primal dual active set method (PDAS) \cite{FanJiaoLu:2014,JiaoJinLu:2015} to find the root of $F(Z)$.
{
\begin{algorithm}
   \caption{PDAS: $x_\lambda\leftarrow \textsl{pdas}(y,\Psi,\lambda,x^0,\textsl{MaxIter}) $}\label{alg:genew}
   \begin{algorithmic}[1]
     \STATE Input $y,\Psi,\lambda$,  initial guess $x^0$, maximum number of iteration \textsl{MaxIter}. Let $d^0= \Psi^t( y -\Psi x^0)/m$.
     \FOR {$k=0,1,...\textsl{MaxIter}$}
     \STATE Compute the active and inactive sets $\calA_k$ and $\calI_k$ respectively by
         \begin{equation*}
           \calA_k = \{i\in [n]\big| |x^{k}_i+ d^{k}_i| >\lambda\}\quad\mbox{and}\quad \calI_k = \overline{\mathcal{A}_k}.
         \end{equation*}
     \STATE Update $x^{k+1}$ and $d^{k+1}$  by
       \begin{equation*}
        \left\{
          \begin{array}{l}
           x_{\mathcal{I}_{k}}^{k+1} = \textbf{0}, \\[1.2ex]
           d_{\mathcal{A}_{k}}^{k+1} =\lambda sign(x_{\mathcal{A}_k}^k+ d_{\mathcal{A}_k}^k), \\[1.2ex]
           (\Psi_{\mathcal{A}_{k}}^t\Psi_{\mathcal{A}_{k}}) x_{\mathcal{A}_{k}}^{k+1} =\Psi_{\mathcal{A}_{k}}^t(y - m d_{\mathcal{A}_{k}}^{k+1}),  \\[1.2ex]
           d_{\mathcal{I}_{k}}^{k+1} = \Psi_{\mathcal{I}_{k}}^t(y-\Psi_{\mathcal{A}_{k}} x_{\mathcal{A}_{k}}^{k+1})/m.
          \end{array}\right.
       \end{equation*}
     \STATE If $\mathcal{A}_{k} = \mathcal{A}_{k+1}$, stop.
     \ENDFOR
   \STATE Output $x_{\lambda}$.
   \end{algorithmic}
\end{algorithm}
}
\begin{remark}
  We can  stop when $k$ is greater  than   a user-predefined  \textsl{MaxIter}. Since    the PDAS  converges after  one iteration,  a desirable property stated in  Theorem \ref{thm:cov}, we set  $\textsl{MaxIter} = 1$.
\end{remark}

The PDAS is actually a generalized Newton method for finding roots of nonsmooth equations \cite{ItoKunisch:2008,QiSun:1993}, since the iteration in Algorithm \ref{alg:genew} can be equivalently  reformulated as
\begin{align}
 \quad & J^k D^k  = -  F(Z^k), \label{newton1}\\
\quad & Z^{k+1} = Z^k + D^k,\label{newton2}
\end{align}
where
\begin{equation}\label{Jacobi}
J^k = \left(
        \begin{array}{cc}
            J^k_{1} & J^k_{2}\\
            \Psi^t\Psi & m\textbf{I} \\
          \end{array}
        \right)
, \quad
J^k_{1} = \left(
          \begin{array}{cc}
            \textbf{0}_{{\cal A}_k,{\cal A}_k} & \textbf{0}_{{\cal A}_k,{\cal I}_k} \\
            \textbf{0}_{{\cal I}_k,{\cal A}_k}& \textbf{I}_{{\cal I}_k,{\cal I}_k}\\
          \end{array}
        \right)
 \quad \textrm{and} \quad
  J^k_{2} = \left(
          \begin{array}{cc}
            -I_{{\cal A}_k,{\cal A}_k} & \textbf{0}_{{\cal A}_k,{\cal I}_k}\\
            \textbf{0}_{{\cal I}_k,{\cal A}_k}& \textbf{0}_{{\cal I}_k,{\cal I}_k}\\
          \end{array}
        \right).
 \end{equation}
We prove  this equivalency   in  Appendix \ref{app:shownewton} for completeness.

Local superlinear convergence has been established for generalized Newton methods for nonsmooth equations \cite{ItoKunisch:2008,QiSun:1993}.
  The  PDAS require one iteration to  convergence. We state the results here for completeness, which is proved in   \cite{FanJiaoLu:2014}.

\begin{theorem}\label{thm:cov}
Let $x$ and $d$ satisfy    \eqref{kkteq}. Denote  $\mathcal{\widetilde {A}} =
\{i\in [n] \big | |x_i + d_i|\geq \lambda \}$  and
  $\omega = \min_{i \in [n], |x_i + d_i| \neq \lambda } \big \{||x_i + d_i| - \lambda|\}.$
  Let $x^0$ and  $d^0 = \Psi^t(y - \Psi x^0)/m$ be  initial input of   algorithm \ref{alg:genew}.
If  the columns of $\Psi_{\mathcal{\widetilde {A}}} $ are   full-rank and  the initial input  satisfies
$\|x - x^0\|_{\infty} + \|d - d^0\|_{\infty} \leq \omega.$
Then,
 $x^1 = x$, where $x^1$ is    updated  from  $x^0$ after one iteration.
\end{theorem}


\subsection{Globalization and automatic regularization parameter selection}

To  apply the PDAS  (Algorithm \ref{alg:genew}) to  \eqref{subreg}, we need to have an  initial guess $x^0$ and specify a proper regularization parameter $\lambda$ in $\textsl{pdas}(y,\Psi,\lambda,x^0,\textsl{MaxIter})$. In this section, we address  these two issues together with continuation.  Since the PDAS is a Newton type algorithm with fast   local  convergence rate and $x_{\ell_1}$ is piecewise linear  function of $\lambda$ \cite{Osborne:2000},    we adopt a
continuation   to fully exploit the fast local convergence. In particular,
this is a simple way to globalize the convergence of PDAS \cite{FanJiaoLu:2014}.
   Observing that $x = \textbf{0}$ satisfies  \eqref{kkteq} if $\lambda \geq \lambda_0 = \|\Psi^t y/m\|_{\infty}$, we define  $\lambda_t = \lambda_0 \rho^t$ with  $\rho\in(0,1)$ for $t=1, 2, \ldots$.  We
 run  Algorithm \ref{alg:genew} on the sequence $\{\lambda_t\}_{t}$ with warmstart, i.e., using  the solution $x_{\lambda_t}$
as an initial guess for the $\lambda_{t+1}$-problem. When the whole continuation   is done we obtain a solution path of \eqref{subreg}. For simplicity, we refer to the
 PDAS with continuation as PDASC described in Algorithm \ref{alg:hotgenew}.

\begin{algorithm}[h!]
   \caption{PDASC:  {$\{x_{\lambda_t}\}_{t\in [\textsl{MaxGrid}]}\leftarrow \textsl{pdasc}(y,\Psi,\lambda_0,x_{\lambda_0},\rho,\textsl{MaxGrid},\textsl{MaxIter})$}}\label{alg:hotgenew}
   \begin{algorithmic}[1]
   \STATE Input, $y, \Psi$, $\lambda_0 = \|\Psi^t y/m\|_{\infty}$, $x_{\lambda_0} = \textbf{0}$, $\rho \in (0,1)$, $\textsl{MaxGrid}$, \textsl{MaxIter}.
     \FOR {$t=1,2,...\textsl{MaxGrid}$}
     \STATE Run algorithm \ref{alg:genew} $x_{\lambda_t}\leftarrow\textsl{pdas}(y,\Psi,\lambda,x^0,\textsl{MaxIter})$   with $\lambda=\lambda_t = \rho^t\lambda_0$, initialized with $x^0 = x_{\lambda_{t-1}}$.
     \STATE Check the stopping criterion.
     \ENDFOR
     \STATE Output, $\{x_{\lambda_t}\}_{t\in [\textsl{MaxGrid}]}$.
   \end{algorithmic}
\end{algorithm}

The regularization parameter
$\lambda$ in the $\ell_1$-regularized  1-bit CS  model \eqref{subreg}, which compromises the tradeoff between  data fidelity and
the sparsity level of the solution, is important for theoretical analysis and practical computation.
 However, the  well known regularization parameter selection rules such as discrepancy principle \cite{Engle:1996,ItoJin:2014}, balancing principle  \cite{ClasonJinKunisch:2010,JinZhou:2012,ClasonBangtiKunisch:2011,ItoJinTakeuchi:2011} or Bayesian information criterion  \cite{FanJiaoLu:2014,Konishi:2008}, are not applicable to the 1-bit CS problem considered here, since the model errors are not available directly. 
Here we propose a  novel rule  to select regularization  parameter  automatically.
We  run the   PDASC to yield a solution path until
$\|x_{\lambda_T}\|_0 > \lfloor \frac{m}{\log n} \rfloor$
for the smallest possible   $T$.
Let $S_{\ell} = \{\lambda_t:
  \|x_{\lambda_t}\|_0  = \ell, t = 1,..., T \}, \ell  =1,...,\lfloor \frac{m}{\log n} \rfloor$ be the set of regularization parameter at which the output of PDAS has $\ell$ nonzero elements.
  We determine $\lambda$ by voting, i.e.,
  \begin{equation}\label{autoreg}
  \hat{\lambda} = \max \{S_{\bar{\ell}}\} \quad  \textrm{and} \quad \bar{\ell } = \arg\max_{\ell}\{|S_{\ell}|\}.
  \end{equation}
   Therefore, our parameter selection  rule is  seamlessly integrated with the continuation  strategy which serves as a globalization technique
 without any extra computational overhead.

Here we give an example to show the accuracy of our proposed regularization parameter selection rule \eqref{autoreg} with data $\{m = 400, n = 10^3,s = 5,\nu = 0.5,\sigma= 0.01,q = 2.5\%\}$. Descriptions   of the data  can be found  in Section \ref{num}.
 Left panel of Fig. \ref{fig:aspath} shows the size of active set $\|x_{\lambda_t}\|_0$ along the path of PDASC and right panel shows the
  underlying true signal $x^*$ and the solution $x_{\hat{\lambda}}$ selected by  \eqref{autoreg}.

\begin{figure}[htb!]
  \centering
   \begin{tabular}{cc}
   \includegraphics[trim = 0cm 0cm 0cm 0cm, clip=true,width=.4\textwidth]{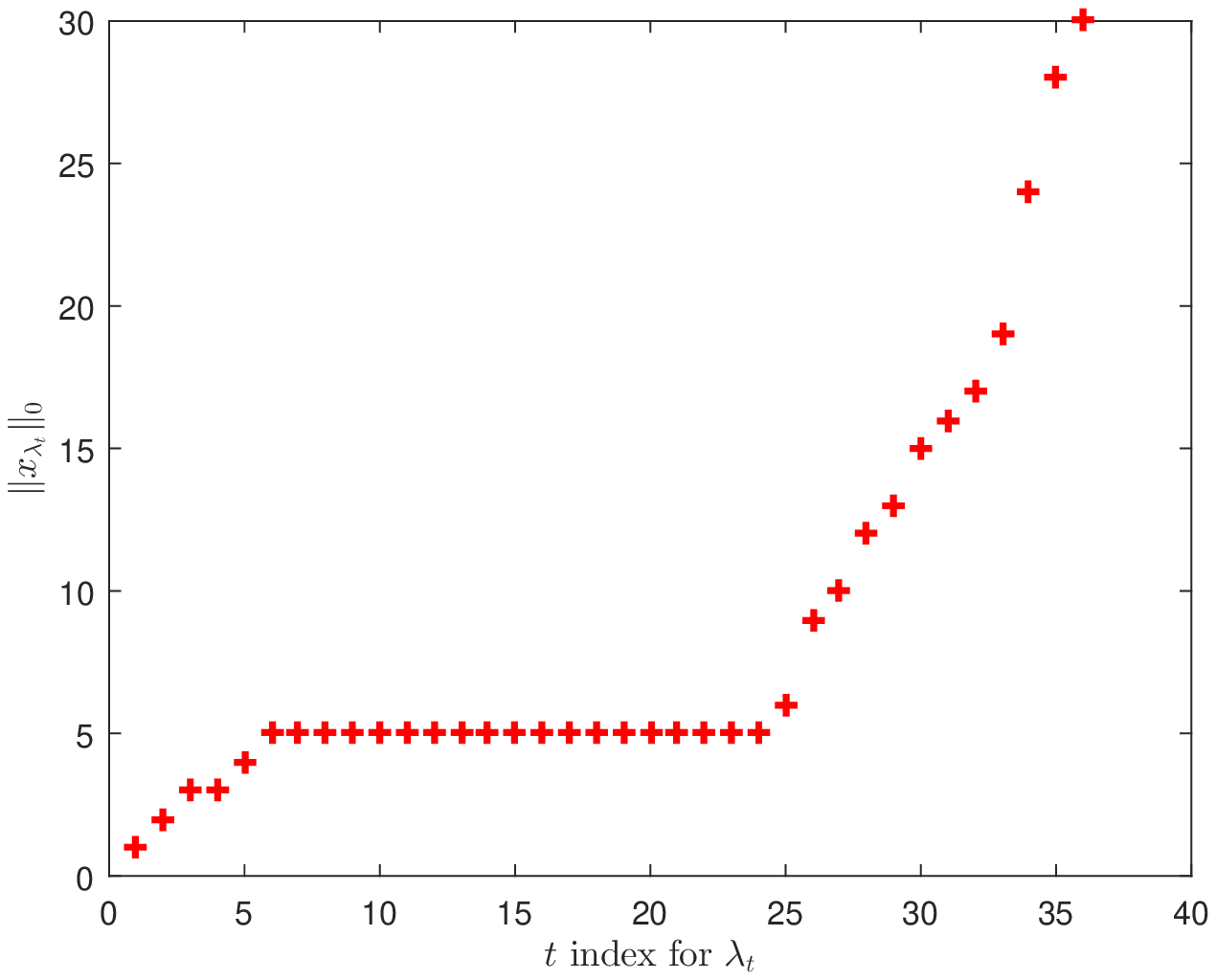}&
   \includegraphics[trim = 0cm 0cm 0cm 0cm, clip=true,width=.4\textwidth]{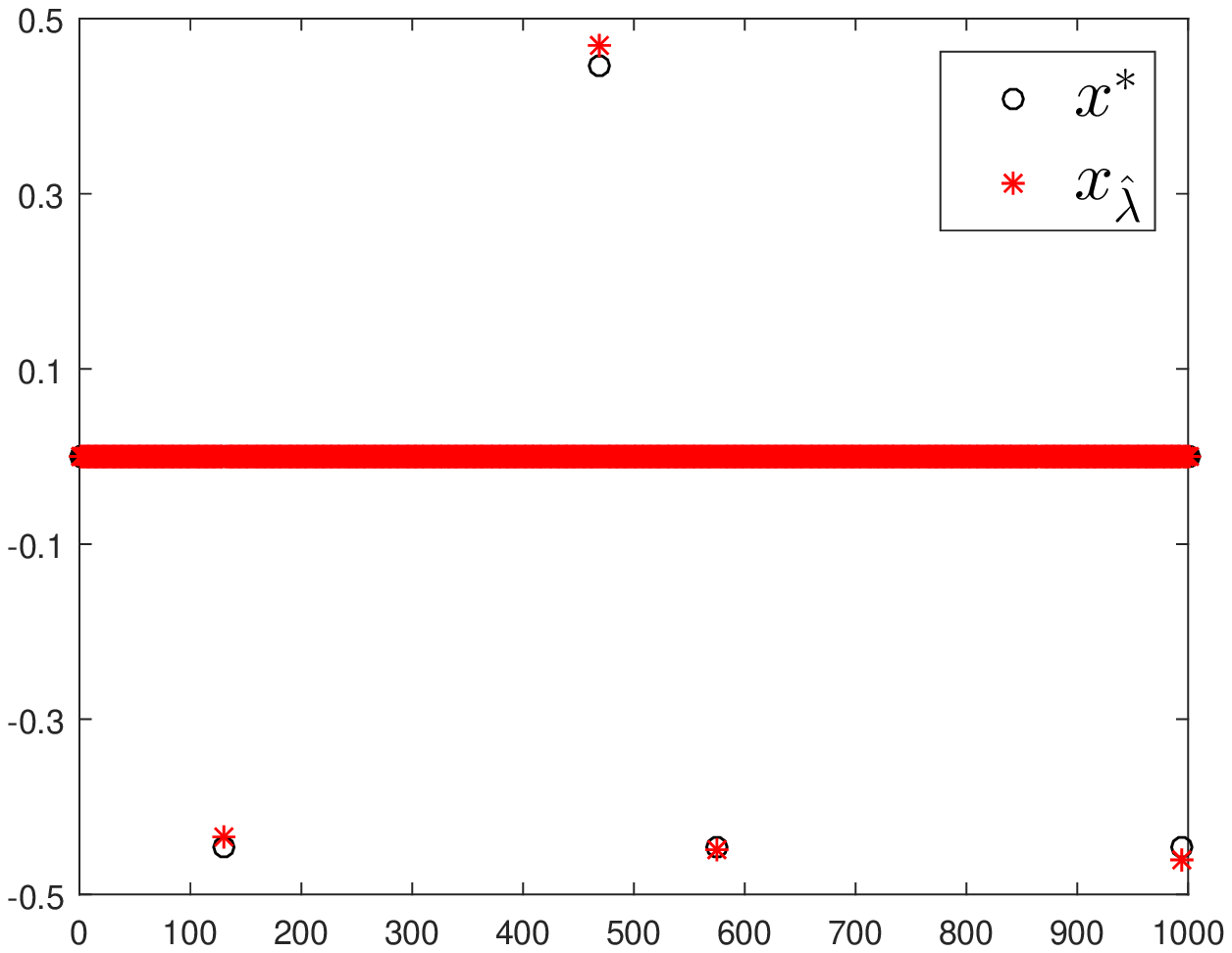}
   \end{tabular}
   \caption{Active set size on the path (left panel)  on data  $\{m = 400, n = 10^3,s = 5,\nu = 0.5,\sigma= 0.01,q = 2.5\%\}$ and the underlying true signal $x^*$ and the solution $x_{\hat{\lambda}}$  selected by \eqref{autoreg} (right panel).}\label{fig:aspath}
\end{figure}


\section{Numerical simulation}\label{num}

In this section we showcase the performance of our proposed least square  decoders \eqref{ls} and \eqref{subreg}.
All the computations were performed on a four-core laptop with 2.90 GHz and 8 GB RAM using \texttt{MATLAB} 2015b.
The \texttt{MATLAB} package \texttt{1-bitPDASC} for reproducing all the numerical results can be found at
\url{http://faculty.zuel.edu.cn/tjyjxxy/jyl/list.htm}.
\subsection{Experiment setup}
First we describe the  data generation process and  our parameter choice. In all numerical examples the underlying target  signal $x^*$ with
$\|x^*\|_0 = s$ is given, and the observation  $y$ is generated by
 $y = \eta \odot\textrm{sign} (\Psi x^* + \epsilon)$,
where the rows of $\Psi$ are iid samples  from  $\mathcal{N}(\textbf{0},\Sigma)$ with $\Sigma_{jk} = \nu^{|j-k|}, 1 \le j, k \le n$. We keep the convention $0^{0} = 1.$ The elements of  $\epsilon$ are generated from $\mathcal{N}(\textbf{0},\textbf{I}_m)$,
$\eta\in \mathcal{R}^{m}$ has independent coordinate $\eta_i$
with $\mathbb{P}[\eta_i = 1] = 1- \mathbb{P}[\eta_i = -1] = q$.  Here, we use $\{m,n,s,\nu,\sigma,q\}$ to denote the data generated as   above  for short. We fix $\rho = 0.95,\textsl{MaxGrid} = 200,\textsl{MaxIter} = 1$ in our proposed  PDASC algorithm and use \eqref{autoreg}  to determine regularization parameter $\lambda$.  All the  simulation results   are based on 100 independent replications.

\subsection{Accuracy and Robustness of $x_{ls}$ when $m>n$}

Now we present numerical results to illustrate the accuracy of the least square decoder $x_{ls}$ and its robustness to the noise and the sign flips.
 Fig. \ref{fig:lowsigma} shows the recovery error $\|x_{ls}-x^*\|$ on data set $\{m = 10^3,n = 10,s = 10,\nu = 0.3,\sigma= 0:0.05:0.5,q = 2.5\%\}$. Left panel of Fig. \ref{fig:lowq}  shows the recovery error $\|x_{ls}-x^*\|$ on data set
  $\{m = 1000,n = 10,s = 10,\nu = 0.3,\sigma= 0.01,q = 0:1\%:10\%\}$ and right  panel gives recovery error $\|x_{ls}+x^*\|$ on data $\{m = 1000,n = 10,s = 10,\nu = 0.3,\sigma= 0.01,q = 90:1\%:100\%\}$.
  It is observed that the recovery error $\|x_{ls}-x^*\|$ ($\|x_{ls}+x^*\|$) of the  least square decoder  is small (around 0.1)  and  robust to  noise level $\sigma$  and sign flips probability $q$. This confirms theoretically investigations    in Theorem \ref{errls}, which states the error is of order  $\widetilde{\mathcal{O}}(\sqrt{\frac{n}{m}}) = 0.1$.

\begin{figure}[htb!]
  \centering
   \includegraphics[trim = 0cm 0cm 0cm 0cm, clip=true,width=.4\textwidth]{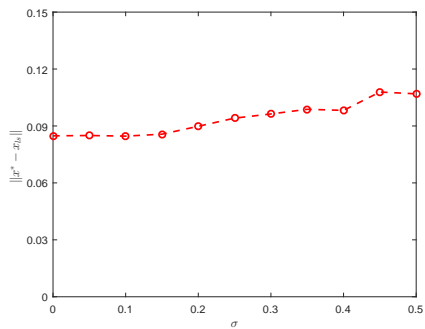}
   \caption{Recovery error  $\|x_{ls}-x^*\|$ v.s. $\sigma$ on $\{m = 1000,n = 10,s = 10,\nu = 0.3,\sigma= 0:0.05:0.5,q = 2.5\%\}$}\label{fig:lowsigma}
\end{figure}

\begin{figure}[htb!]
  \centering
  \begin{tabular}{cc}
   \includegraphics[trim = 0cm 0cm 0cm 0cm, clip=true,width=.4\textwidth]{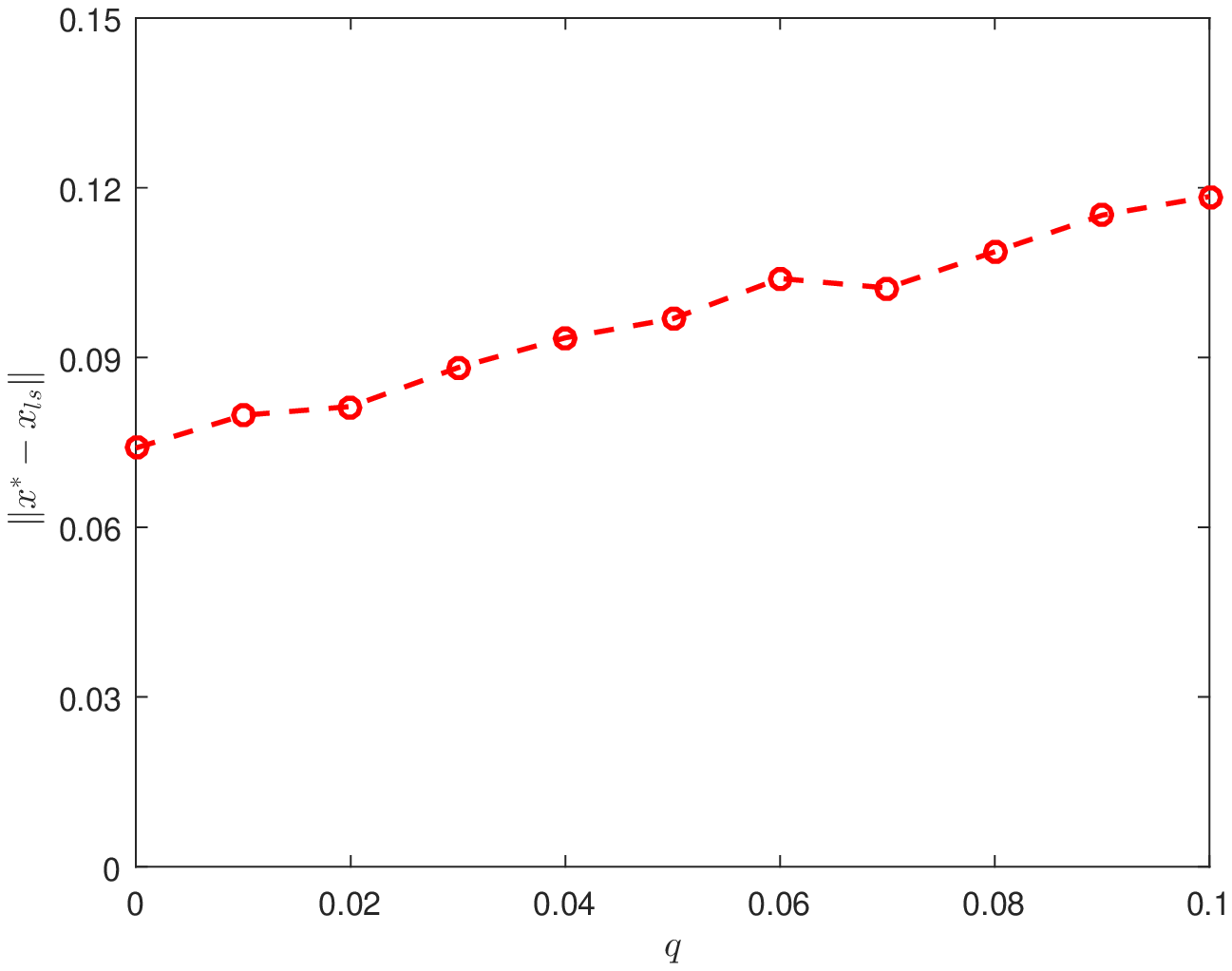}&
   \includegraphics[trim = 0cm 0cm 0cm 0cm, clip=true,width=.4\textwidth]{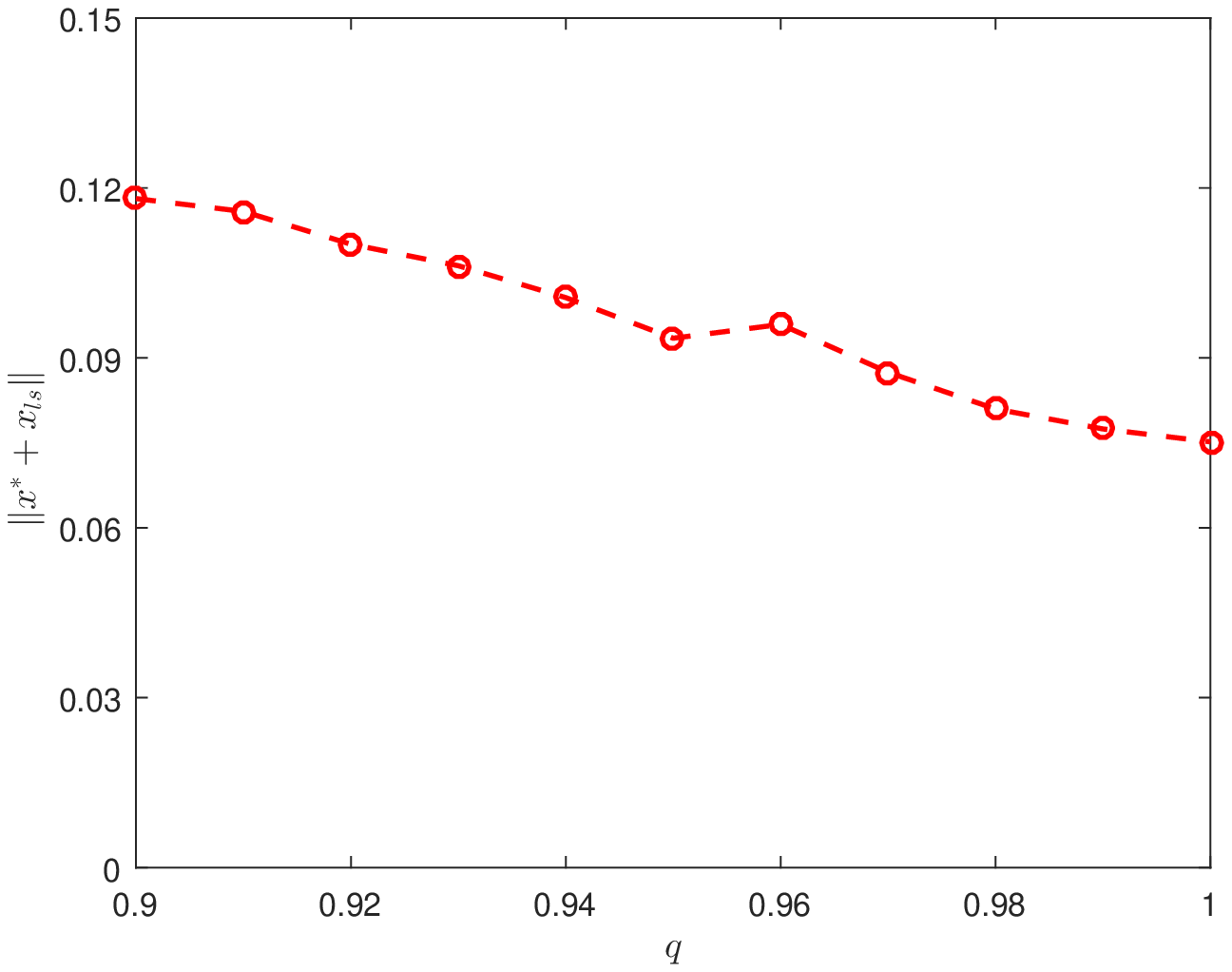}\\
  \end{tabular}
   \caption{Recovery error  $\|x_{ls}-x^*\|$ v.s. $q$ on $\{m = 10^3,n = 10,s = 10,\nu = 0.3,\sigma= 0.01,q = 0:1\%:10\%\}$ (left panel) and $\|-x_{ls}-x^*\|$ on  $\{m = 1000,n = 10,s = 10,\nu = 0.3,\sigma= 0.01,q = 90\%:1\%:100\%\}$ (right panel).}\label{fig:lowq}
\end{figure}

\subsection{Support recovery of  $x_{\ell_1}$ when $m<n$}
We conduct  simulations to illustrate the performance of model \eqref{subreg} PDASC  algorithm. We report how the exact support recovery probability varies with the sparsity level $s$, the noise level $\sigma$ and the probability $q$ of sign flips.
 Fig. \ref{fig:probsupp}  indicates that, as long as the sparsity level $s$ is not large, $x_{\ell_1}$ recovers the  underlying true support
with high probability     even if the measurement  contains noise and is corrupted by  sign flips. This confirms the theoretical investigations in Theorem \ref{errsub}.
\begin{figure}[htb!]
  \centering
  \begin{tabular}{cc}
   \includegraphics[trim = 0cm 0cm 0cm 0cm, clip=true,width=.4\textwidth]{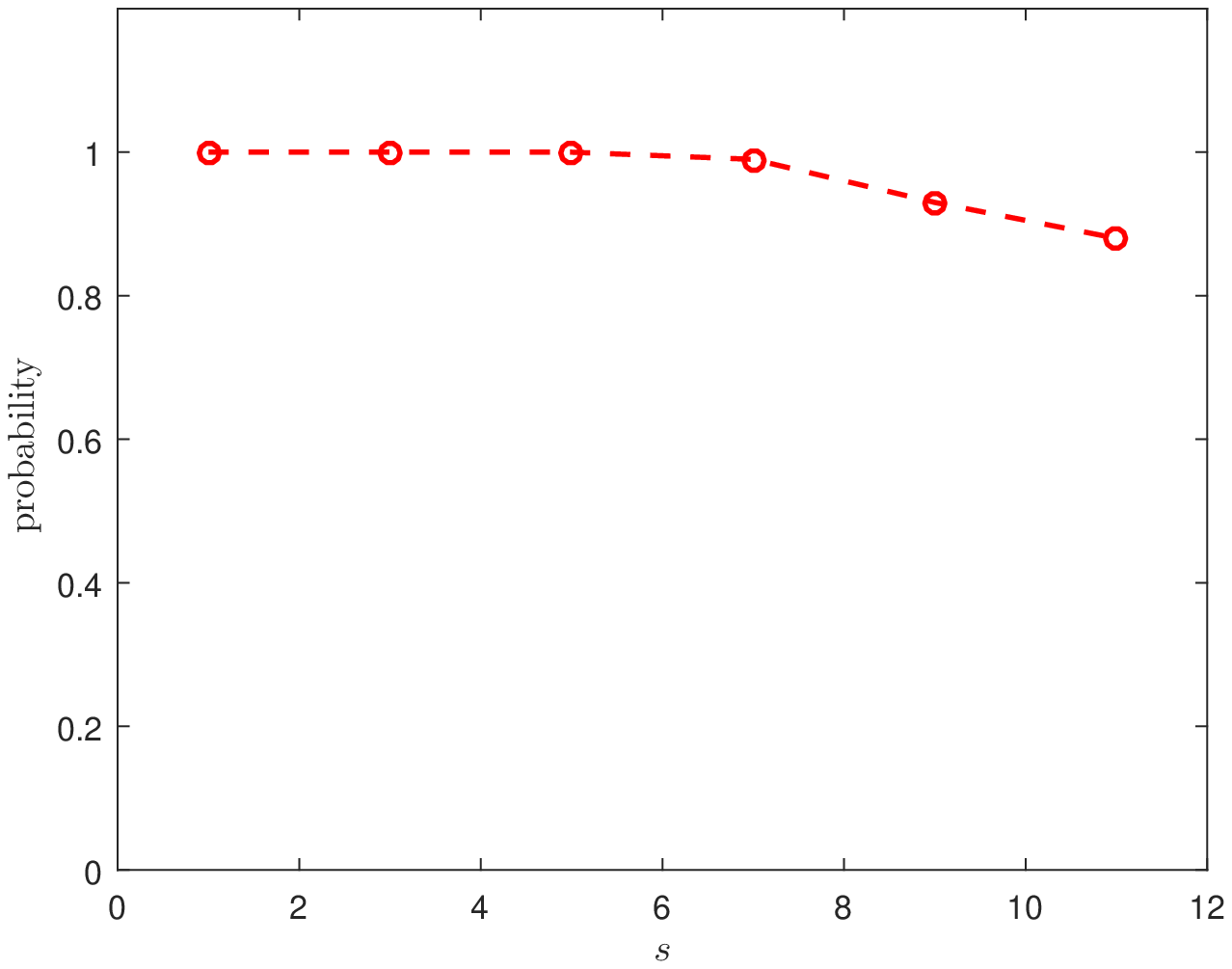} &
   \includegraphics[trim = 0cm 0cm 0cm 0cm, clip=true,width=.4\textwidth]{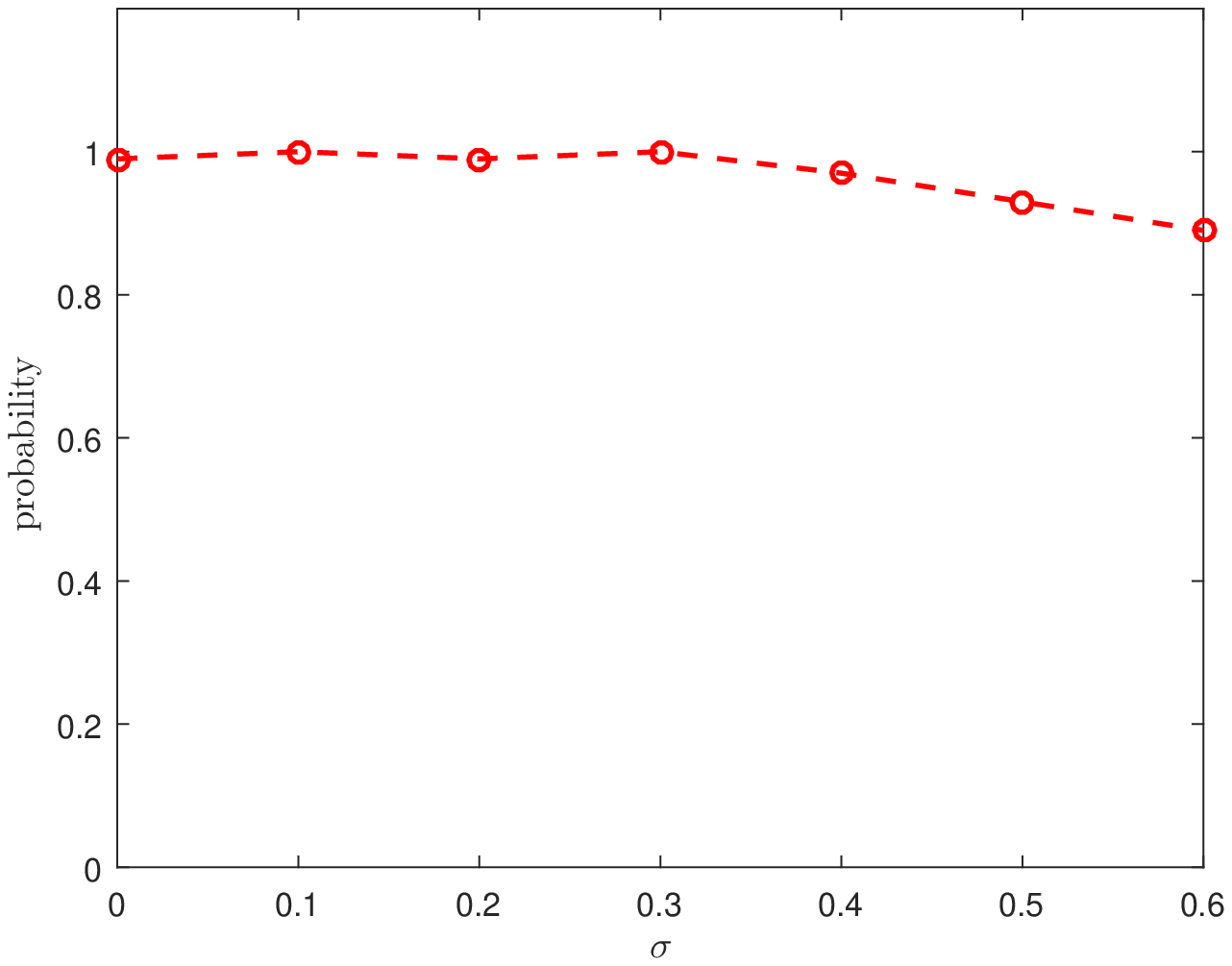} \\
   (a) $s$ & (b) $\sigma$ \\
   \includegraphics[trim = 0cm 0cm 0cm 0cm, clip=true,width=.4\textwidth]{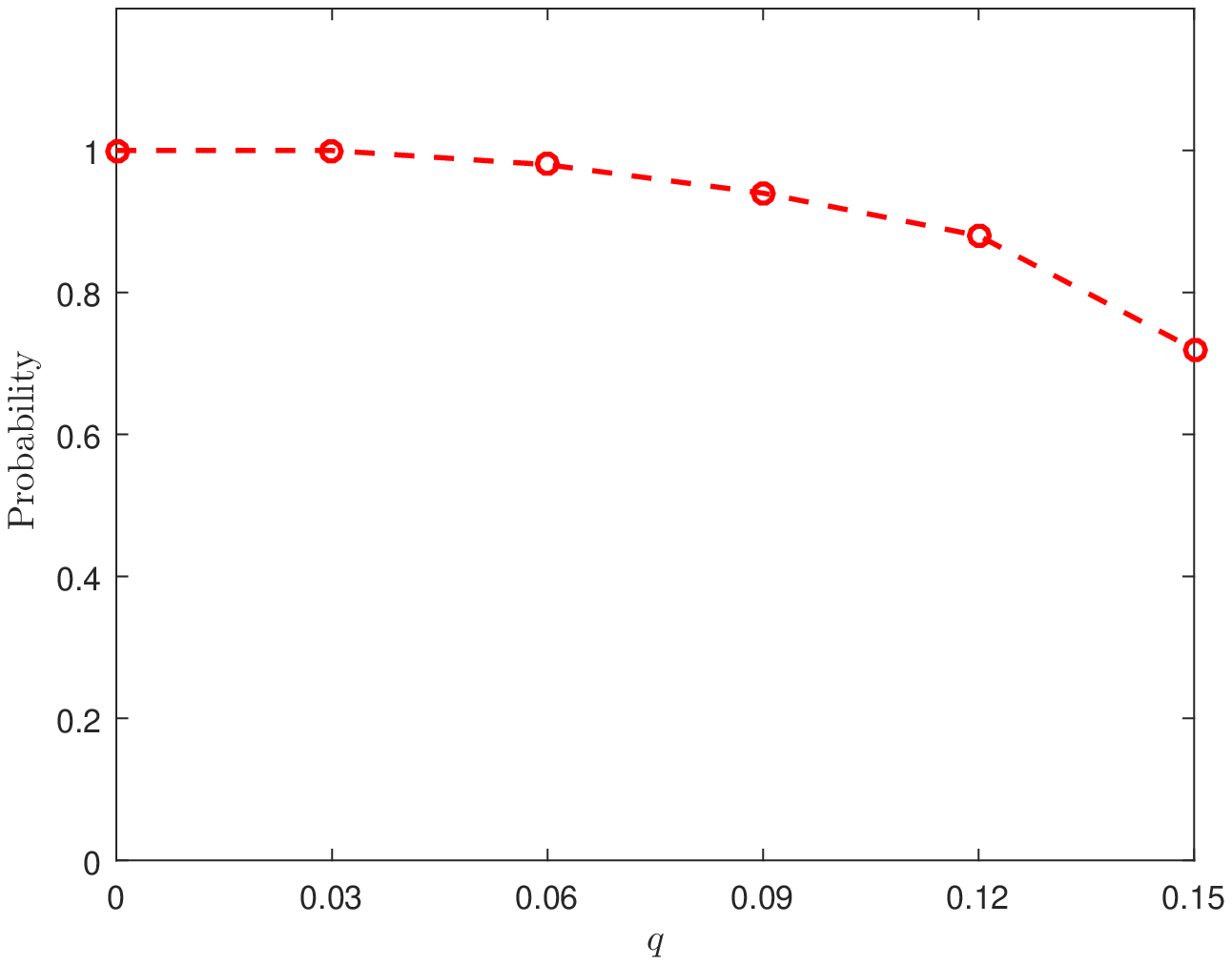} &
   \includegraphics[trim = 0cm 0cm 0cm 0cm, clip=true,width=.4\textwidth]{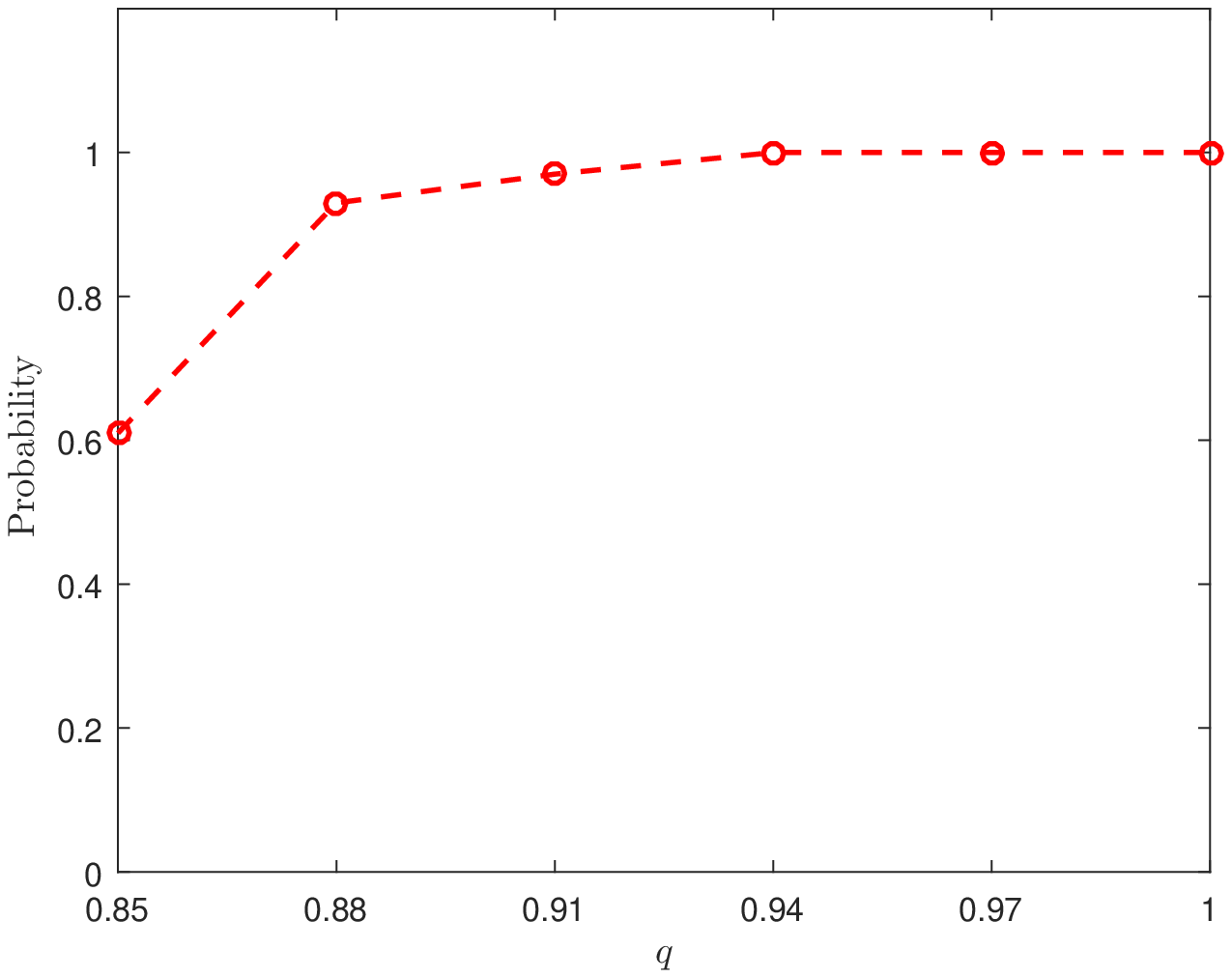}\\
     \\ (c) $q$ & (d) $q$
  \end{tabular}
   \caption{The exact support recovery probability v.s. $s$, $\sigma$ and $q$ on data set  $\{m = 500,n = 1000,s = 1:2:12,\nu = 0.1,\sigma= 0.05,q = 1\%\}$ (panel (a)), $\{m = 500,n = 1000,s = 5,\nu = 0.3,\sigma= 0:0.1:0.6, q = 5\%\}$ (panel (b)) and $\{m = 500,n = 1000,s = 5,\nu = 0.1,\sigma= 0.01,q = 0:3\%:15\%\}$ (panel (c)), $\{m = 500,n = 1000,s = 5,\nu = 0.1,\sigma= 0.01,q = 85\%:3\%:100\%\}$  (panel (d)). }\label{fig:probsupp}
\end{figure}

\subsection{Comparison with other state-of-the-art}
Now we compare our proposed model \eqref{subreg} and  PDASC algorithm with several state-of-the-art methods such as BIHT  \cite{JacquesDegraux:2013} (\url{http://perso.uclouvain.be/laurent.jacques/index.php/Main/BIHTDemo}), AOP \cite{YanYangOsher:2012} and PBAOP \cite{HuangShiYan:2015} (both AOP and PBAOP  available  at  \url{http://www.esat.kuleuven.be/stadius/ADB/huang/downloads/1bitCSLab.zip}) and  linear projection (LP) \cite{Vershynin:2015,PlanVershynin:2017}.   
 BIHT, AOP, LP and  PBAOP  are all required  to specify  the true sparsity level $s$. Both AOP and PBAOP also need to required  to specify the sign flips probability $q$.
  The PDASC   does not  require  to specify   the unknown parameter sparsity level $s$ or the probability of sign flips $q$.
We use $\{m = 500,n = 1000,s = 5,\nu = 0.1,\sigma= 0, q = 0\}$, $\{m = 500,n = 1000,s = 5,\nu = 0.3,\sigma= 0.3, q = 5\%\}$, $\{m = 500,n = 1000,s = 5,\nu = 0.5,\sigma= 0.5, q = 10\%\}$,          and $\{m = 800,n = 2000,s = 10,\nu = 0.1,\sigma= 0.1, q = 1\%\}$, $\{m = 800,n = 2000,s = 10,\nu = 0.2,\sigma= 0.3, q = 3\%\}$, $\{m = 800,n = 2000,s = 10,\nu = 0.3,\sigma= 0.5, q = 5\%\}$, and $\{m = 5000,n = 20000,s = 50,\nu = 0,\sigma= 0.2, q = 3\}$, $\{m = 5000,n = 20000,s = 50,\nu = 0,\sigma= 0.1, q = 1\}$, $\{m = 5000,n = 20000,s = 5,\nu = 0,\sigma= 0.3, q = 5\%\}$.
 The average CPU time in seconds (Time (s)), the  average of the  $\ell_2$  error $\|x_{\ell} - x^*\|$ ($\ell_2$-Err), and the  probability of  exactly  recovering  true support (PrE ($\%$)) are reported in Table \ref{tab:compother}.
  The  PDASC is comparatively very  fast and the most  accurate even if  the correlation $\nu$, the noise level $\sigma$ and the  probability of sign flips $q$ are large.
\begin{table}
  \caption{Comparison PDASC with state-of-the-art methods on  CPU time in seconds (Time (s)), average $\ell_2$  error $\|x_{\ell} - x^*\|$ ($\ell_2$-Err), probability on exactly  recovering of  true support (PrE ($\%$)).}
 \label{tab:compother}
  \vspace{-0.3cm}
  \begin{center}
  \scalebox{.85}{
  \begin{tabular}{cccccccccccc}
   \hline \hline
  \multicolumn{11}{c}{\quad \quad\quad$\{m = 500,n = 1000,s = 5\}$}\\
 \hline
  \multicolumn{4}{c}{\quad \quad \quad \quad (a) $\{\nu = 0.1,\sigma= 0.1, q = 1\%\}$}
  &&\multicolumn{3}{c}{(b) $\{\nu = 0.3,\sigma= 0.3, q = 5\%\}$}&&\multicolumn{3}{c}{(c) $\{\nu = 0.1,\sigma= 0.5, q = 10\%\}$}\\
  \cline{2-4}  \cline{6-8} \cline{10-12}
  Method           &Time (s)         & $\ell_2$-Err   & PrE $(\%)$&   & Time (s)  & $\ell_2$-Err    & PrE $(\%)$   &  & Time       &$\ell_2$-Err    &PrE               \\
   BIHT          &      1.42e-1     &1.89e-1         & 92  &   &     1.31e-1 &   5.73e-1 &        19  &          & 1.32e-1 &    9.39e-1   &   0\\
   AOP           &      2.72e-1     &7.29e-2         & 100  &   &     3.55e-1 &   2.11e-1 &        92  &          & 3.58e-1 &    4.22e-1   &    44 \\
   LP            &      8.70e-3     &4.19e-1         & 98 &   &      8.50e-3 &   4.22e-1 &        93  &          & 8.30e-3 &    4.81e-1   &    26\\
   PBAOP         &      1.46e-1    &9.08e-2         & 100 &   &      1.36e-1 &   2.05e-1 &        90  &          & 1.35e-1 &    4.53e-1   &    36\\
   PDASC         &      4.11e-2     &6.77e-2         & 100 &   &      4.38e-2 &   9.40e-2 &        99  &          & 4.56e-2 &    2.21e-1   &    71\\
  \hline
  \hline
  \multicolumn{11}{c}{\quad \quad\quad$\{m = 800,n = 2000,s = 10\}$}\\
 \hline
  \multicolumn{4}{c}{\quad \quad \quad \quad (a) $\{\nu = 0.1,\sigma= 0.1, q = 1\%\}$}
  &&\multicolumn{3}{c}{(b) $\{\nu = 0.3,\sigma= 0.2, q = 3\%\}$}&&\multicolumn{3}{c}{(c) $\{\nu = 0.5,\sigma= 0.3, q = 5\%\}$}\\
  \cline{2-4}  \cline{6-8} \cline{10-12}
  Method           &Time (s)         & $\ell_2$-Err   & PrE $(\%)$&   & Time (s)  & $\ell_2$-Err    & PrE $(\%)$   &  & Time      &$\ell_2$-Err    &PrE               \\
  BIHT          &     4.17e-1     &2.10-1       &  84  &   &           4.25e-1 &        4.21e-1 &            25 &          &    4.35e-1 &    6.46e-1&    0 \\
   AOP           &     1.09e-0     &7.78-2         &  100  &   &         1.10e-0 &       1.76e-1 &            95 &          &    1.16e-0 &    2.86e-1&    59\\
   LP            &     1.95e-2     &4.54-1         &  85 &   &           1.99e-2 &       4.49e-1 &            71 &          &    2.05e-2 &    5.03e-1&    16\\
   PBAOP         &     4.22e-1     &1.00-1         &  100 &   &          4.27e-1 &       1.58e-1 &            99 &          &    4.31e-1 &    2.99e-1&    51\\
   PDASC         &     1.23e-1     &8.66-2         &  100 &   &          1.27e-1 &       1.04e-1 &            98 &          &    1.30e-2 &    1.51e-1&    78\\
  \hline
  \hline
  \multicolumn{11}{c}{\quad \quad\quad$\{m = 5000,n = 20000,s = 50,\nu = 0\}$}\\
 \hline
  \multicolumn{4}{c}{\quad \quad \quad \quad (a) $\{\sigma= 0.1, q = 1\%\}$}
  &&\multicolumn{3}{c}{(b) $\{\sigma= 0.2, q = 3\%\}$}&&\multicolumn{3}{c}{(c) $\{\sigma= 0.3, q = 5\%\}$}\\
  \cline{2-4}  \cline{6-8} \cline{10-12}
  Method           &Time (s)         & $\ell_2$-Err   & PrE $(\%)$&   & Time (s)  & $\ell_2$-Err    & PrE $(\%)$   &  & Time       &$\ell_2$-Err    &PrE               \\
   BIHT          &     2.56e+1     &2.16e-1         & 58  &   &        2.58e+1&      4.54e-1 &            0  &          & 2.58e+1&     6.29e-1   &  0\\
   AOP           &     6.44e+1     &7.56e-2         & 100  &   &       6.46e+1&      1.66e-1 &           96  &          & 6.47e+1&        2.57e-1   &   16 \\
   LP            &     2.35e-1     &4.47e-1         & 38 &   &         2.30e-1&      4.46e-1 &           34  &          & 2.30e-1&        4.47e-1   &   26\\
   PBAOP         &     2.56e+1     &9.89e-2         & 100 &   &        2.58e+1&      1.66e-1 &           95  &          & 2.58e+1&        2.60e-1   &   18\\
   PDASC         &     7.09e-0     &7.97e-2         & 100 &   &        7.17e-0&     9.17e-2 &           99  &           & 7.23e-0&        1.23e-1   &   86\\
  \hline
  \hline
  \end{tabular}}
  \end{center}
\end{table}

Now we compare the PDASC with  the aforementioned competitors   to
recover a one-dimensional  signal.
The true  signal  are sparse  under
 wavelet basis ``Db1" \cite{Mallat:2008}. Thus,  the matrix $\Psi$ is of size $2500\times 8000$ and consists
of   random Gaussian  and an inverse of  one  level Harr wavelet transform.
The target coefficient  has $36$ nonzeros. We set  $\sigma= 0.5$, $q=4\%$. The recovered results are shown in
Fig. \ref{fig:1d} and Table \ref{tab:1d}. The reconstruction by the  PHDAS
is  visually more appealing than others, as shown in  Fig. \ref{fig:1d}.  This is further confirmed by the PSNR value reported in Table \ref{tab:1d}, which is   defined by
$\mathrm{PSNR}=10\cdot \log\frac{V^2}{\rm MSE}$,
where $V$ is the maximum absolute value of the true signal, and MSE is the mean
squared error of the reconstruction.

\begin{figure}[ht!]
  \centering
  \begin{tabular}{cc}
    \includegraphics[trim = 0cm 0cm 0cm 0cm, clip=true,width=5cm]{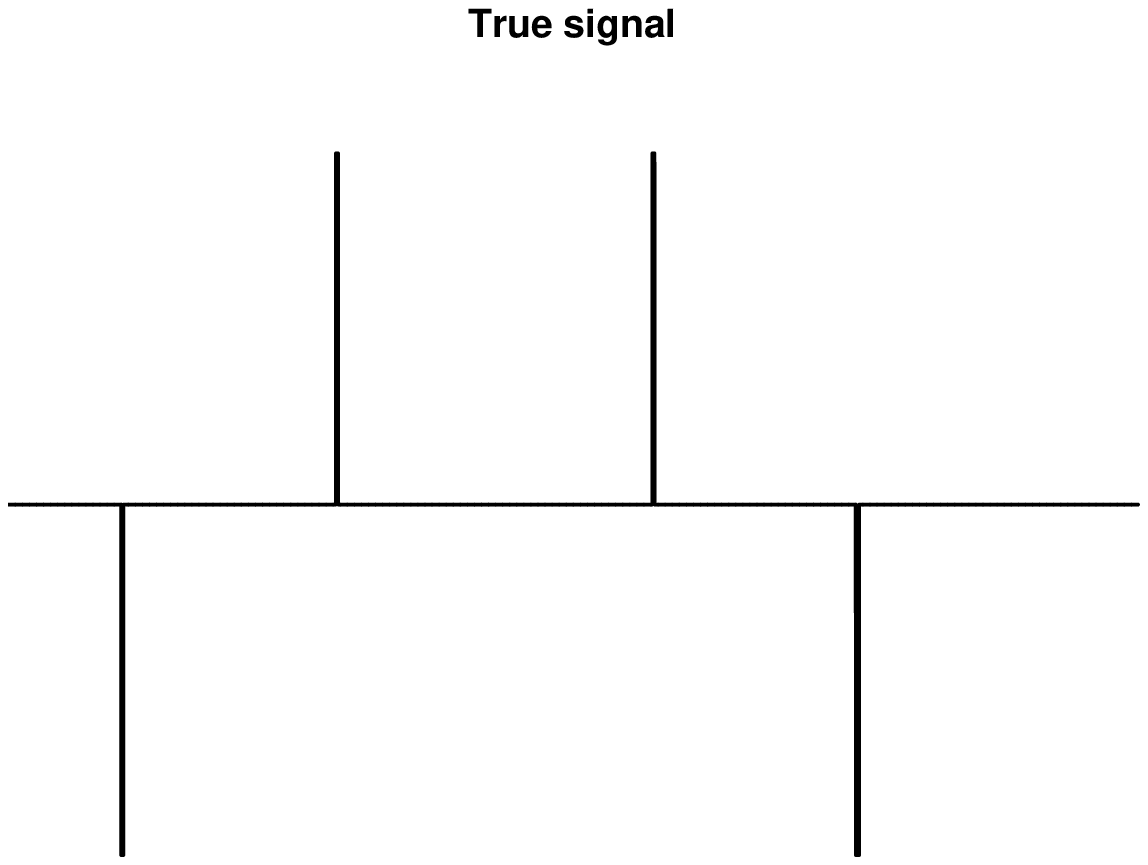} &
    \includegraphics[trim = 0cm 0cm 0cm 0cm, clip=true,width=5cm]{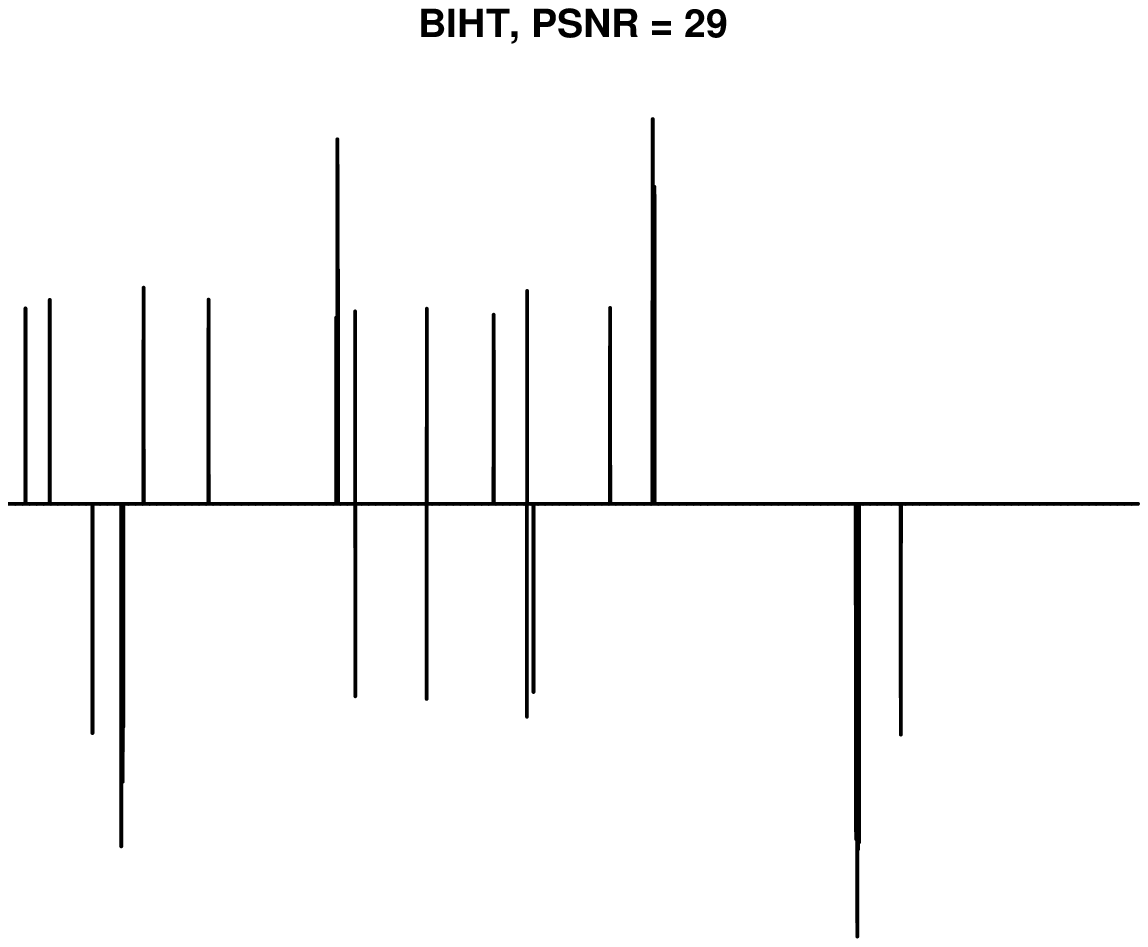}\\
    \includegraphics[trim = 0cm 0cm 0cm 0cm, clip=true,width=5cm]{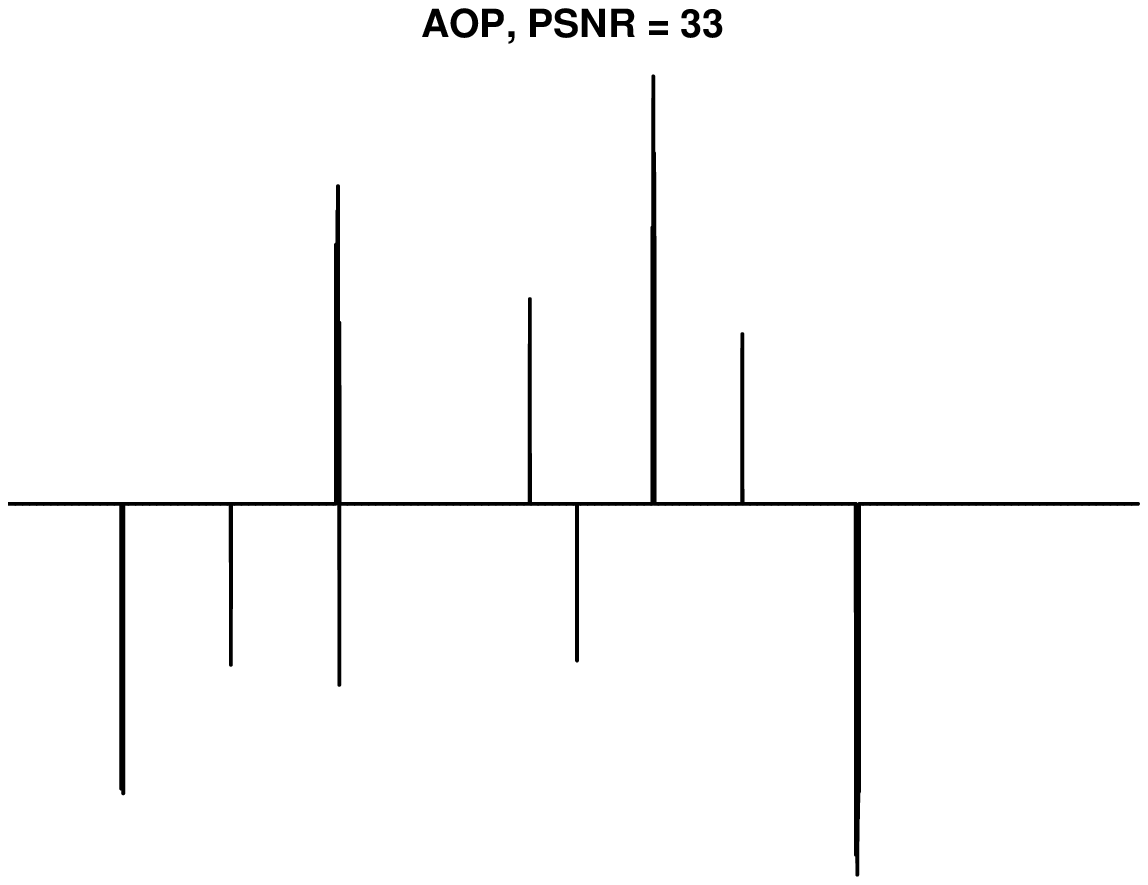} &
    \includegraphics[trim = 0cm 0cm 0cm 0cm, clip=true,width=5cm]{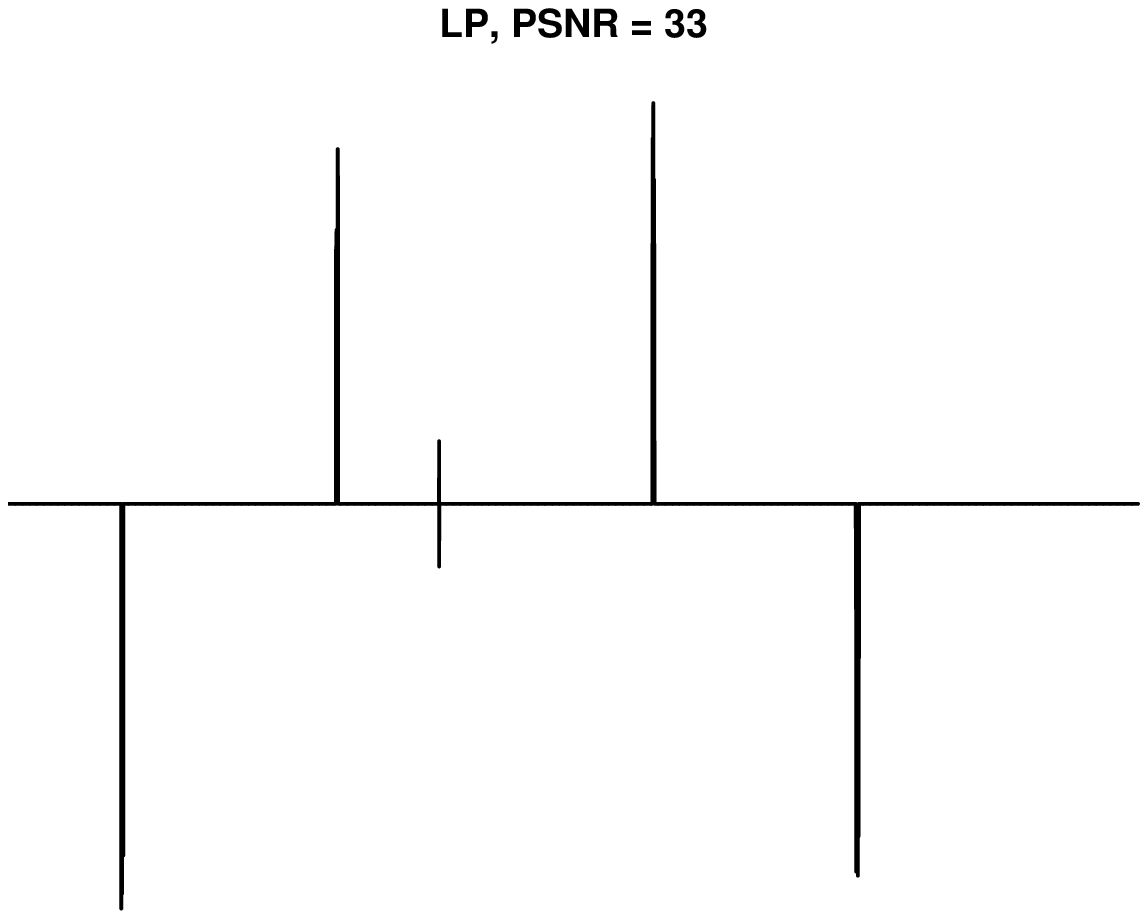}\\
    \includegraphics[trim = 0cm 0cm 0cm 0cm, clip=true,width=5cm]{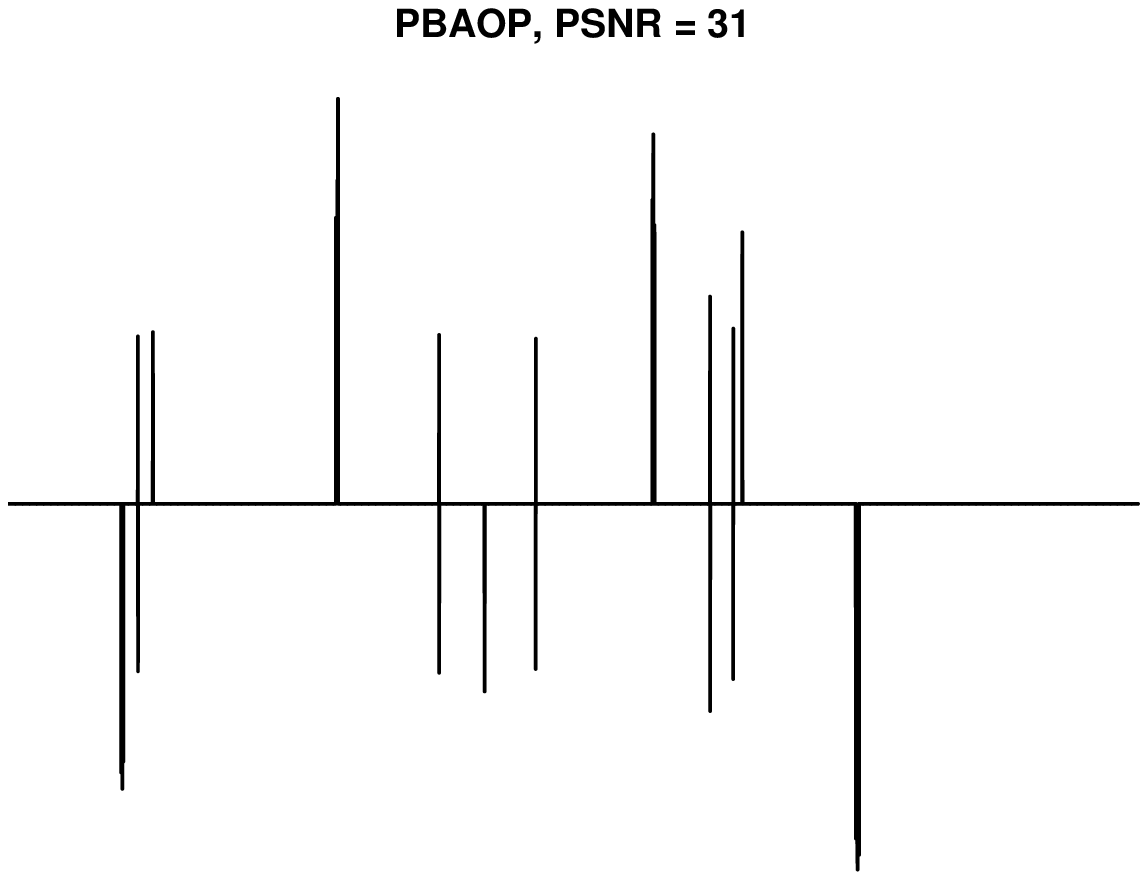} &
    \includegraphics[trim = 0cm 0cm 0cm 0cm, clip=true,width=5cm]{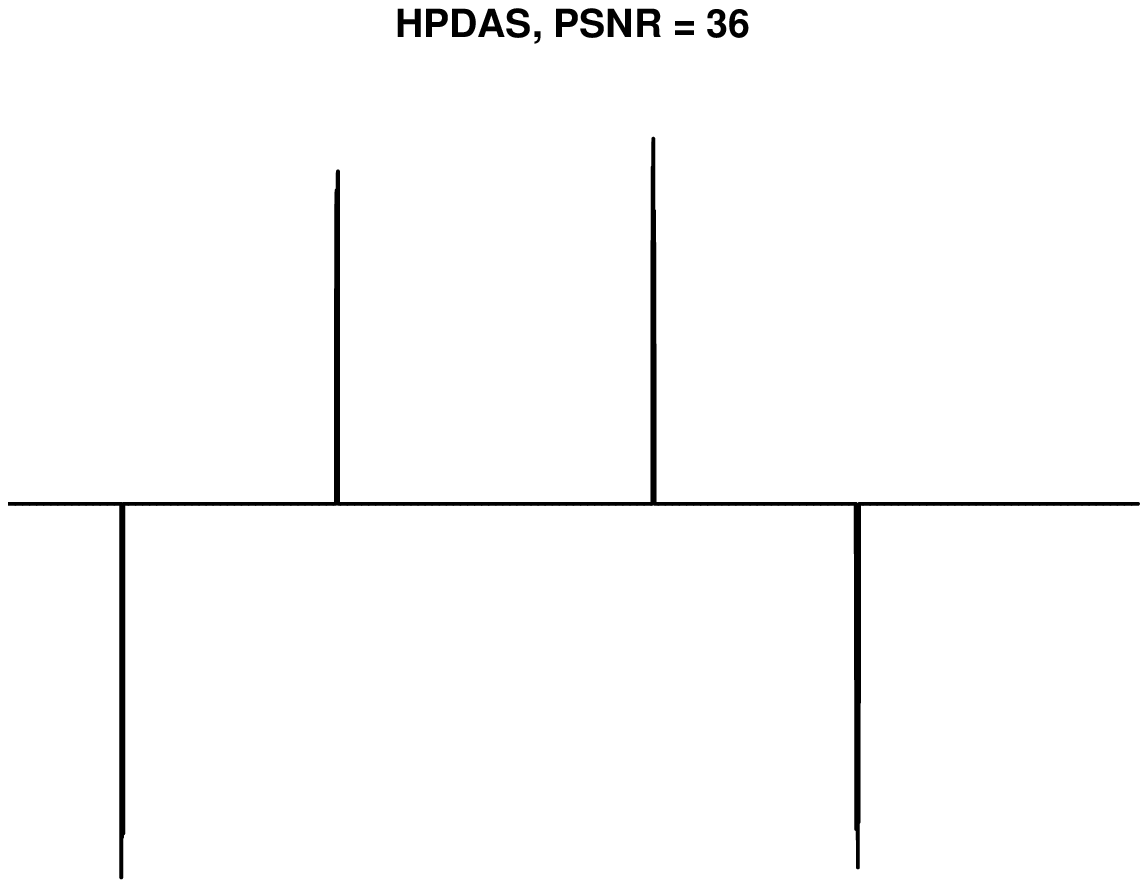}\\
  \end{tabular}
  \caption{Reconstruction  of the one-dimension signal with  $\{m=2500, n=8000, s=36,\nu = 0, \sigma=0.5, q = 4\%\}$.}\label{fig:1d}
\end{figure}

\begin{table}[ht!]
\centering
 \caption{The CPU time in seconds  and the PSNR of one dimensional signal recovery with  $\{m=2500, n=8000, s=36,\nu = 0, \sigma=0.5, q = 4\%\}$.}\label{tab:1d}
 \begin{tabular}{ccccc}
 \hline
    method   &CPU time (s)  &PSNR      \\
 \hline
  BIHT         & 4.97  &29     \\
  AOP          & 4.98  &33     \\
  LP           & 0.11  &33     \\
  PBAOP        & 4.93  &31    \\
  PDASC        &3.26  &36     \\
  \hline
  \end{tabular}
\end{table}

\section{Conclusions}
 In this paper we consider decoding from  1-bit measurements with noise and  sign flips.     For  $m>n$, we show that, up to a constant $c$, with high probability the least squares solution $x_{\textrm{ls}}$ approximates  $ x^*$ with precision $\delta$ as long  as $m \geq\widetilde{\mathcal{O}}(\frac{n}{\delta^2})$.
For $m< n$, we assume  that the underlying target $x^*$ is $s$-sparse,
  and prove that up to a constant $c$, with high probability, the  $\ell_1$-regularized least  squares solution $x_{\ell_1}$ lies in the  ball with center $x^*$  and  radius $\delta$,  provided that  $m \geq \mathcal{O}( \frac{s\log n}{\delta^2})$.
We introduce the   one-step convergent PDAS
method to minimize  the nonsmooth objection function.
We propose a novel tuning  parameter selection rule which is  seamlessly integrated with the  continuation strategy  without any extra computational overhead. Numerical experiments are presented to illustrate salient features of the
model and the efficiency and accuracy of the  algorithm.

There are several avenues for further study. First, many practitioners observed that nonconvex sparse regularization often brings in additional benefit
in the standard CS setting. Whether  the  theoretical and computational results  derived in this paper can still be justified when nonconvex regularizers  are used deserves  further  consideration.
The 1-bit CS is a kind of nonlinear sampling approach.    Analysis of some  other  nonlinear  sampling methods are also of immense interest.

\section*{Acknowledgements}

The research of Y. Jiao is supported by National Science Foundation of  China (NSFC) No. 11501579 and
National Science Foundation of  Hubei Province No. 2016CFB486.  The research of  X. Lu is supported by NSFC Nos. 11471253 and 91630313, and
the research of L. Zhu  is supported by NSFC  No. 11731011 and Chinese Ministry of Education Project of Key  Research Institute of Humanities and Social Sciences at Universities No. 16JJD910002 and National Youth Top-notch Talent Support Program of  China.
\appendix
\section{Proof of Lemma \ref{parallel}}\label{app:parallel}
\begin{proof}
Let $u = \tilde{\psi}^t x^*$. Then $u\sim \mathcal{N}(0,1)$ due to  $\tilde{\psi} \sim \mathcal{N}(\textbf{0},\Sigma)$  and the assumption that $\|x^*\|_{\Sigma} = 1.$
\begin{align*}
   \mathbb{E}[\tilde{\psi} \tilde{y}] &= \mathbb{E}[\tilde{\psi}\tilde{ \eta} \textrm{sign}(\tilde{\psi}^t x^*+\tilde{\epsilon})]
=\mathbb{E}[\tilde{\eta}]  \mathbb{E}[\tilde{\psi} \textrm{sign}(\tilde{\psi}^t x^*+\tilde{\epsilon})] \\
 & = [q - (1-q)]\mathbb{E}[\tilde{\psi} \textrm{sign}(\tilde{\psi}^t x^*+\tilde{\epsilon})] \\
 & = (2q - 1)\mathbb{E}[ \mathbb{E}[ \tilde{\psi}\textrm{sign}(\tilde{\psi}^t x^*+\tilde{\epsilon}) | \tilde{\psi}^t x^*]]\\
 & = (2q - 1)\mathbb{E}[ \mathbb{E}[ \tilde{\psi}| \tilde{\psi}^t x^*] \textrm{sign}(\tilde{\psi}^t x^*+\tilde{\epsilon}) ]\\
 & = (2q - 1) \mathbb{E}[\mathbb{E}[\tilde{\psi} | \tilde{\psi}^t x^*] \mathbb{E}[\textrm{sign}(\tilde{\psi}^t x^*+\tilde{\epsilon}) | \tilde{\psi}^t x^*]]\\
  & = (2q - 1) \mathbb{E}[\Sigma x^*\tilde{\psi}^t x^* \mathbb{E}[\textrm{sign}(\tilde{\psi}^t x^*+\tilde{\epsilon}) | \tilde{\psi}^t x^*]]\\
 & = (2q - 1) \mathbb{E}[\textrm{sign}(\tilde{\psi}^t x^*+\tilde{\epsilon})\tilde{\psi}^t x^*]\Sigma x^*\\
 & = c \Sigma x^*,
\end{align*}
where $c = (2q - 1) \mathbb{E}[\textrm{sign}(\tilde{\psi}^t x^*+\tilde{\epsilon})\tilde{\psi} x^*]$. The second line follows from independence assumption and the third   from law of  total expectation, and the
fourth  and fifth lines are due to
the independence between  $\tilde{\epsilon}$ and $u$,
 and the sixth  line uses   the projection interpretation of conditional expectation i.e., $ \mathbb{E}[\tilde{\psi}|\tilde{\psi}^t x^*] = \frac{\mathbb{E}[\tilde{\psi} \tilde{\psi}^t x^*]-\mathbb{E}[\tilde{\psi}^t x^*]\mathbb{E}[\tilde{\psi}]}{\mathbb{E}[(\tilde{\psi}^t x^* -\mathbb{E}[\tilde{\psi}^t x^*])^2]} (\tilde{\psi}^t x^*-\mathbb{E}[\tilde{\psi}^t x^*]) + \mathbb{E}[\tilde{\psi}]  = \Sigma x^*\tilde{\psi}^t x^*$, where we use $\mathbb{E}[\tilde{\psi}]= \textbf{0}$ and $u\sim \mathcal{N}(0,1)$. Let $f(t) = \frac{1}{\sqrt{2\pi}\sigma} \exp{\frac{-t^2}{2\sigma^2}}$ be  the density function of $\tilde{\epsilon} \sim \mathcal{N}(0,\sigma^2)$.
Integrating  by parts  shows that
\begin{align*}
 c = (2q - 1) \mathbb{E}[\textrm{sign}(u+\tilde{\epsilon})u] &= (2q - 1)\mathbb{E}[(1-2 \mathbb{P}[\tilde{\epsilon}\leq -u])u] \\
 &= (2q - 1)\mathbb{E}[\frac{\partial(1-2 \mathbb{P}[\tilde{\epsilon}\leq -u)}{\partial u}] \\
 & = (2q - 1)\mathbb{E}[f(-u)] = (2q - 1)\sqrt{\frac{2}{\pi(\sigma^2+1)}}.
 \end{align*}
The proof is completed by inverting $c\Sigma$.
\end{proof}

\section{Preliminaries}
We 
%
recall some simple properties of subgaussian and subexponential random variables.

\begin{lemma}\label{basic}( Lemma 2.7.7 of \cite{Vershynin:2016} and  Remark 5.18 of  \cite{Vershynin:2010}.)
Let $\xi_1$ and $\xi_2$ be subgaussian random variables. Then both $\xi_1\xi_2$ and $ \xi_1\xi_2- \mathbb{E}[\xi_1\xi_2]$ are subexponential random variables.
\end{lemma}

Lemma B.2 states the   nonasymptotic bound on the spectrums of $\Psi$ and the operator norm of $\Psi^t\Psi/m - \Sigma$ when $m \geq \mathcal{O}(n).$
\begin{lemma}\label{vershynin5.3940}
Let $\Psi\in \mathcal{R}^{m\times n}$  whose rows $\psi_i^t$ are independent subgaussian  vectors  in $\mathcal{R}^n$ with mean $\textbf{0}$ and covariance matrix $\Sigma$. Let $m > n$. Then for every $t>0$ with
probability at least $1-2\exp{(-C_1 t^2)}$, one has
\begin{equation}\label{norm1}
 (1- \tau)\sqrt{\gamma_{min}(\Sigma)} \leq \sqrt{\gamma_{\textrm{min}}(\frac{\Psi^t\Psi}{m})}\leq \sqrt{\gamma_{\textrm{max}}(\frac{\Psi^t\Psi}{m})} \leq (1 +\tau)\sqrt{\gamma_{max}(\Sigma)},
\end{equation}
and
\begin{equation}\label{norm2}
 \|\Psi^t\Psi/m - \Sigma\|\leq \max\{\tau,\tau^2\} \gamma_{\textrm{max}}(\Sigma),
 \end{equation}
 where $\tau = C_2 \sqrt{\frac{n}{m}} + \frac{t}{\sqrt{m}},$
 and  $C_1, C_2$ are generic positive constants depending   on  the
 maximum subgaussian norm of rows of $\Psi$.
\end{lemma}
\begin{proof}
Let $\Phi = \Psi \Sigma^{-\frac{1}{2}}$. Then the rows of $\Phi$ are independent sub-gaussian isotropic vectors.   \eqref{norm1} follows from Theorem 5.39 and Lemma 5.36 of \cite{Vershynin:2010} and
\eqref{norm2} is a direct consequence of  Remark 5.40 of \cite{Vershynin:2010}.

\end{proof}

We  state the Bernstein-type  inequality for the sum of independent and mean $0$ sub-exponential random   random variables.

\begin{lemma}\label{vershynin5.17}(Corollary  5.17 of \cite{Vershynin:2010})
Let $\xi_1,...,\xi_m$ be independent centered sub-exponential random variables.   Then for every $t>0$  one has $$\mathbb{P}[|\sum_{i = 1}^m \xi_i|/m \geq t] \leq  2 \exp (- \min\{C_1t^2,C_2t\}m)$$ where $C_1$ and $C_2$ are are generic positive constants depending   on  the maximum
 subexponential  norm of  of $\xi_i$.
\end{lemma}


\begin{lemma}\label{usedaf}
Let $\Psi\in \mathcal{R}^{m\times n}$  whose rows $\psi_i$ are independent subgaussian  vectors  in $\mathcal{R}^{n\times 1}$ with mean $\textbf{0}$ and covariance matrix $\Sigma$.  Then,
with probability at least $1-2\exp{(-C_1 C_2^2n)}$,
\begin{equation}\label{l2noise1}
\|\Psi^t\Psi/m  -  \mathbb{E}[\Psi^t\Psi/m]\| \leq 2 C_2 \gamma_{\textrm{max}}(\Sigma) \sqrt{\frac{n}{m}},
\end{equation}
as long as  $m \geq 4C^2_2 n$. Furthermore,
if $m> \frac{4C_1}{C_2^2}\log n$, then
\begin{equation}\label{linfnoise2}
\|\sum_{i=1}^m (\mathbb{E}[\psi_i y_i]-\psi_i y_i)/m\|_{\infty} \leq 2\sqrt{\frac{\log n}{C_1 m}},
\end{equation}
holds with probability at least $1-\frac{2}{n^3},$
and
\begin{equation}\label{linfnoise3}
\|\Psi^t \Psi/m - \Sigma\|_{\infty} \leq 2\sqrt{\frac{\log n}{C_1 m}},
\end{equation}
holds with probability at least $1-\frac{1}{n^2}.$
\end{lemma}
\begin{proof}
  By \eqref{ub1}, $\mathbb{E}[\Psi^t\Psi/m] =\Sigma$, hence  \eqref{l2noise1} follows  from   \eqref{norm2} with   $t = C_2 \sqrt{n}$  and the assumption $m \geq 4C^2_2n$.
Define $ G_{i,j} := y_i (\psi_i)_j\in \mathcal{R}^{1}, i= 1,...,m, j=1,..,n,$ which is subexponential by Lemma \ref{basic}.
Therefore,
\begin{align*}
\mathbb{P} [\|\sum_{i=1}^m (\mathbb{E}[\psi_i y_i]-\psi_i y_i)/m\|_{\infty}\geq  t ] &= \mathbb{P}[\bigcup_{j= 1}^{n}\{|\sum_{i = 1}^{m} G_{i,j}/m| \geq t \}]\\
&\leq \sum_{j= 1}^{n} \mathbb{P}[|\sum_{i = 1}^{m} G^{(i)}_{k,j}|/m \geq t] \\
&\leq n \exp (- \min\{C_1t^2,C_2t\}m)\\
& \leq 2n \exp (- C_1t^2m),
\end{align*}
where the first inequality is due to the union bound, the second follows from Lemma \ref{vershynin5.17} and the last  is because of
restrictions $t\leq \frac{C_2}{C_1}$ and $m<n.$
Then \eqref{linfnoise2} follows from our assumption that $m> \frac{4C_1}{C_2^2}\log n$ by setting  $t = 2\sqrt{\frac{\log n}{C_1 m}}$ and.
Let
 $ G_{j,k}^{i} := (\psi_i)_{j}(\psi_{i})_{k} - \Sigma_{j,k}\in \mathcal{R}^{1}, i = 1,...,n, j= 1,...,n, \ell=1,..,n,$ which is mean $0$ subexponential by Lemma \ref{basic}.
Therefore,
\begin{align}
\mathbb{P}[\|\Psi^t \Psi/m - \Sigma\|_{\infty}\geq t] &= \mathbb{P}[\max_{j,k} |\sum_{i=1}^{m}G_{j,k}^{i} /m|\geq t]\nonumber\\
&= \mathbb{P}[\bigcup_{j= 1,k=1}^{n,n}\{|\sum_{i=1}^{m}G_{j,k}^{i}/m| \geq t \}]\nonumber\\
&\leq \sum_{j= 1, k=1}^{n,n} \mathbb{P}[|\sum_{i=1}^{m}G_{j,k}^{i}/m| \geq t] \nonumber\\
&\leq n^2 \exp (- \min\{C_1t^2,C_2t\}m) \nonumber\\
& \leq n^2 \exp (-  C_1t^2 m), \nonumber
\end{align}
where the first inequality is due to the union bound, and the second follows from   Lemma \ref{vershynin5.17} and the last inequality  is because of
restricting $t\leq \frac{C_2}{C_1}$.
Then by the  assumption that $m> \frac{4C_1}{C_2^2}\log n$, Lemma \ref{linfnoise3} follows  by  setting  $t = 2\sqrt{\frac{\log n}{C_1 m}}$.
\end{proof}

\section{Proof of Theorem \ref{errls}}\label{app:errls}

\begin{proof}
First we show that the sample covariance  matrix   $\Psi^t\Psi/m$ is invertible with probability at least $1- 2\exp{(-C_1 C_2^2 n)}$   as long as
$m > 4 C_2 n.$ This follows  from \eqref{norm1} in Lemma \ref{vershynin5.3940} by setting $t = C_2\sqrt{n}$.
Recall
\begin{equation}\label{xstar}
 \widetilde{x}^* = c x^*.
 \end{equation}
 Let
\begin{equation}\label{effniose}
 \Delta =  y -\Psi  \widetilde{x}^*,
 \end{equation}
 be the error in  measuring  nonlinearity, sign flips and  noise in the 1-bit CS measurement.
 Then,
\begin{align}
  \|\Psi^t \Delta/m \|_{2} & = \|\Psi^t (\Psi \widetilde{x}^* - y)/m \|_{2}\nonumber = \|\frac{\Psi^t \Psi}{m}  \widetilde{x}^* - \Psi^t y/m \|_{2} \nonumber = \|(\frac{\Psi^t \Psi}{m}  \widetilde{x}^* - \Sigma  \widetilde{x}^*) +  (\Sigma  \widetilde{x}^*-\Psi^t y/m )\|_{2}\nonumber\\
 & = \|(\frac{\Psi^t \Psi}{m}  \widetilde{x}^* -  \mathbb{E}[\frac{\Psi^t \Psi}{m}  \widetilde{x}^*]  )+ (\mathbb{E}[\Psi^t y/m]-\Psi^t y/m )\|_{2}\nonumber\\
 & \leq  |c|\|{x^*}\|_{2} \|\frac{\Psi^t \Psi}{m}  -  \mathbb{E}[\frac{\Psi^t \Psi}{m}  ]\|  + \|\sum_{i=1}^m (\mathbb{E}[\psi_i y_i]-\psi_i y_i)/m\|_{2}\nonumber\\
 &\leq |c| \frac{1}{\sqrt{\gamma_{\textrm{min}}(\Sigma)}}  \|\frac{\Psi^t \Psi}{m}  -  \mathbb{E}[\frac{\Psi^t \Psi}{m}  ]\| + \sqrt{n} \|\sum_{i=1}^m (\mathbb{E}[\psi_i y_i]-\psi_i y_i)/m\|_{\infty}, \label{noicon}
 \end{align}
 where the fourth  equality  is due to \eqref{parallel}, \eqref{ub1} and \eqref{ub2},  the first inequality   follows from the triangle inequality and the definition of $\widetilde{x}^*$, and the  last inequality uses    the assumption  $1=\|x^*\|^2_{\Sigma} \geq \gamma_{\textrm{min}}(\Sigma) \|x^*\|_{2}^2 $ and the fact that $\|\cdot\|_2\leq \sqrt{n}\|\cdot\|_{\infty} $.  Combining with \eqref{l2noise1} and $\eqref{linfnoise2}$, we deduce that,
 with probability at least $1- 2\exp{(-C_1 C_2^2 n)} -\frac{2}{n^3}$,
 \begin{equation}\label{l2}
 \|\Psi^t \Delta/m\|_{2} \leq  \sqrt{\frac{n}{m}} 2(|c|  C_2 \sqrt{\kappa(\Sigma)\gamma_{\textrm{max}}(\Sigma)}  + \sqrt{\frac{\log n}{C_1}}).
 \end{equation}
Now we  prove that $\|x_{\textrm{ls}}/c- cx^*\|_2 = \widetilde{\mathcal{O}} (\sqrt{\frac{n}{m}} /c)$ with high probability.
\begin{align*}
|c| \|x_{\textrm{ls}}/c- x^*\|_2 &= \|x_{\textrm{ls}}- \widetilde{x}^*\|_2 = \|(\Psi^t\Psi)^{-1}\Psi^t y - \widetilde{x}^*\|_2\\
&\leq \|(\Psi^t\Psi/m)^{-1}\Psi^t (\Psi \widetilde{x}^* + y - \Psi \widetilde{x}^* )/m - \widetilde{x}^*\|_2\\
& = \|(\Psi^t\Psi/m)^{-1}\|\|\Psi^t \Delta/m \|_2\\
& \leq  \sqrt{\frac{n}{m}} 2(|c| C_2 \sqrt{\kappa(\Sigma)\gamma_{\textrm{max}}(\Sigma)}  + \sqrt{\frac{\log n}{C_1}})/(1-2C_2 \sqrt{\frac{n}{m}})^2\\
&\leq  \sqrt{\frac{n}{m}} 4(|c| C_2 \sqrt{\kappa(\Sigma)\gamma_{\textrm{max}}(\Sigma)}  + \sqrt{\frac{\log n}{C_1}}),
\end{align*}
where the second inequality  follows with probability at least $1- 4\exp{(-C_1 C_2^2 n)} -\frac{2}{n^3}$ from \eqref{noicon}  and  \eqref{norm1} by setting $t = C_2\sqrt{n}$, and the last line is due to the assumption $m \geq 16 C_2^2 n.$
Hence, the proof of Theorem \ref{errls} is completed by dividing $|c|$ on both side and some algebra.
\end{proof}

\section{Proof of Theorem \ref{errsub}}\label{app:errsub}
\begin{proof}
Our proof is based on  Lemmas \ref{slev} - \ref{rel} below.
Denote $R =  x_{\ell_1} - \widetilde{x}^* $, $\mathcal{A}^* = supp(x^*)$ and $\mathcal{I}^* = \overline{\mathcal{A}^*}$.
 The first  lemma shows that $ R$ is sparse in the sense that its energy is mainly cumulated on $\mathcal{A}^*$ if $\lambda$ is chosen properly.
\begin{lemma}\label{slev}
Let
 \begin{equation}\label{cone}
 \mathcal{C}_{\mathcal{A^*}}= \{z \in \mathcal{R}^n: \|z_{\mathcal{I^*}}\|_1 \leq 3\|z_{\mathcal{A^*}}\|_1\},
 \end{equation}
 and define  $\mathcal{E} = \{ \|\Psi^t \Delta/m\|_{\infty} \leq \lambda/2 \}$. Conditioning  on the event $\mathcal{E}$, we have  $R\in  \mathcal{C}_{\mathcal{A^*}}$.
  \end{lemma}

\begin{proof}
The optimality of $x_{\ell_1}$ implies that
$\frac{1}{2m}\|y - \Psi x_{\ell_1}\|_2^2 + \lambda \|x_{\ell_1}\|_1\leq \frac{1}{2m}\|y - \Psi \widetilde{x}^*\|_2^2 + \lambda \|\widetilde{x}^*\|_1.$
Recall that  $y =  \Psi \widetilde{x}^* + \Delta$. Some algebra on the above display shows
 \begin{align*}
 &\frac{1}{2m}\|\Psi R\|_2^2 + \lambda \|R_{\mathcal{I}^*}\|_1\leq  \langle R, \Psi^t \Delta/m \rangle + \lambda \|R_{\mathcal{A}^*}\|_1\\
 &\leq \|R\|_1 \|\Psi^t \Delta/m\|_{\infty} + \lambda \|R_{\mathcal{A}^*}\|_1\leq \|R\|_1 \lambda/2 + \lambda \|R_{\mathcal{A}^*}\|_1,
 \end{align*}
 where, we use Cauchy Schwartz inequality and the definition of $\mathcal{E}$. The above inequality shows
 \begin{equation}\label{basiceq}
 \frac{1}{m}\|\Psi R\|_2^2 + \lambda \|R_{\mathcal{I}^*}\|_1 \leq 3\lambda \|R_{\mathcal{A}^*}\|_1,
\end{equation}
i.e., $R \in \mathcal{C}_{\mathcal{A^*}}$. This finishes the proof of Lemma \ref{slev}.
\end{proof}

The next  Lemma gives a lower bound on $\mathbb{P}[\mathcal{E}]$ with a proper regularization parameter $\lambda$.
  \begin{lemma}\label{noiselinf}
  Let $C_3 \geq \|x^*\|_1$.
If  $m> \frac{4C_1}{C_2^2}\log n$, taking $\lambda =  \frac{4(1+|c|C_3)}{\sqrt{C_1}}\sqrt{\frac{\log n}{m}}$, then with probability at least $1-2/n^3  -2/n^2$, one has
  \begin{equation}
\|\Psi^t \Delta/m\|_{\infty} \leq \lambda/2.
  \end{equation}
  \end{lemma}

  \begin{proof}
  \begin{align}
  \|\Psi^t \Delta/m \|_{\infty} & = \|\Psi^t(\Psi \widetilde{x}^* - y)/m \|_{\infty} = \|(\frac{\Psi^t \Psi}{m}  \widetilde{x}^*- \Sigma  \widetilde{x}^*) +  (\Sigma   \widetilde{x}^*-\Psi^t y/m )\|_{\infty}\nonumber\\
 &\leq \|(\frac{\Psi^t \Psi}{m}  \widetilde{x}^* - \Sigma  \widetilde{x}^*)\|_{\infty} + \|(\mathbb{E}[\Psi^t y/m]-\Psi^t y/m )\|_{\infty}\nonumber\\
 & \leq  |c| \|(\frac{\Psi^t \Psi}{m}  - \Sigma) \|_{\infty}\|x^*\|_1  + \|\sum_{i=1}^m (\mathbb{E}[\psi_i y_i]-\psi_i y_i)/m\|_{\infty}\nonumber\\
 &\leq  |c|C_3 2\sqrt{\frac{\log n}{C_1m}} +  2\sqrt{\frac{\log n}{C_1m}}\nonumber\\
 &=  \frac{2(1+|c|C_3)}{\sqrt{C_1}}\sqrt{\frac{\log n}{m}},\nonumber
 \end{align}
 where the    first inequality  is due to the triangle inequality,  \eqref{parallel}, \eqref{ub1} and \eqref{ub2}, the second inequality  follows from the definition of $\widetilde{x}^*$ and Cauchy-Schwartz inequality,  and the third one uses \eqref{linfnoise2} and \eqref{linfnoise3}. The proof of Lemma \ref{noiselinf} is completed by setting $\lambda = \frac{4(1+|c|C_3)}{\sqrt{C_1}}\sqrt{\frac{\log n}{m}}$.
  \end{proof}

The last Lemma shows $\Psi$ is strongly convex along the direction contained  in the cone $\mathcal{C}_{\mathcal{A^*}}$ defined in \eqref{cone}.

\begin{lemma}\label{rel}
If $s \leq \exp^{(1-\frac{C_1}{2})} n$ and ¡¡
 $m \geq \frac{64(4\kappa(\Sigma)+1)^2}{C_1}s\log\frac{en}{s},$
then with probability at least  $1-1/n^{2}$,
we have ¡¡
 $$\|\Psi z \|^2_2/m \geq \frac{\gamma_{min}(\Sigma)}{68(4\kappa(\Sigma)+1)^2} \|z\|_2^2,\quad \forall z\in \mathcal{C}_{\mathcal{A^*}}.$$
  \end{lemma}
\begin{proof}
$\forall z \in \mathcal{C}_{\mathcal{A^*}} = \{v \in \mathcal{R}^n: \|v_{\mathcal{I^*}}\|_1 \leq 3\|v_{\mathcal{A^*}}\|_1\}$, we sort its entries such that $$|z_{k_1}| \geq |z_{k_2}|\geq ...\geq |z_{k_n}|.$$ Let $\mathcal{A} = \{k_1,...,k_s\}$, and $\mathcal{I} = \bigcup_{t \geq 1} \mathcal{I}_{t} = \{k_{ st+1}, ..., k_{(t+1)s}\},$ where $s = \|x^*\|_0.$ Then
\begin{equation}\label{nnc}
\|z_{\mathcal{I}}\|_1 \leq 3\|z_{\mathcal{A}}\|_1.
\end{equation}
By  the elementary  inequality  $\|v\|_2\leq \frac{\|v\|_1}{\sqrt{s}} + \frac{\sqrt{s}\|v\|_{\infty}}{4}, \quad \forall v \in \mathcal{R}^s,$
we have
\begin{align}
\sum_{t\geq 1} \|z_{\mathcal{I}_{t}}\|_2 &\leq \sum_{t \geq 1} \frac{\|z_{\mathcal{I}_{t}}\|_1}{\sqrt{s}} + \frac{\sqrt{s}\|z_{\mathcal{I}_{t}}\|_{\infty}}{4}\nonumber\\
 & = \frac{\|z_{\mathcal{I}}\|_1}{\sqrt{s}} + \sum_{t\geq 1}\frac{\sqrt{s}\|z_{k_{st+1}}\|_{\infty}}{4}\nonumber\\
 &\leq \frac{\|z_{\mathcal{I}}\|_1}{\sqrt{s}} + \frac{\|z_{\mathcal{A}}\|_{1}+\|z_{\mathcal{I}}\|_{1}}{4\sqrt{s}}\nonumber\\
 & \leq \frac{3\|z_{\mathcal{A}}\|_1}{\sqrt{s}} + \frac{\|z_{\mathcal{A}}\|_{1}+3\|z_{\mathcal{A}}\|_{1}}{4\sqrt{s}} \quad (using \quad \eqref{nnc})\nonumber\\
 & =  \frac{4\|z_{\mathcal{A}}\|_1}{\sqrt{s}}\nonumber\\
 & \leq 4 \|z_{\mathcal{A}}\|_2,\label{shift}
\end{align}
which implies,
\begin{equation}\label{sw}
\|z\|_2^2 = \|z_{\mathcal{A}}\|_2^2 +   \sum_{t\geq 1} \|z_{\mathcal{I}_{t}}\|_2^2 \leq \|z_{\mathcal{A}}\|_2^2 +(\sum_{t\geq 1} \|z_{\mathcal{I}_{t}}\|_2)^2 = 17 \|z_{\mathcal{A}}\|_2^2.
\end{equation}
Define  $$C_{2s}(\Psi) = \inf_{A\subset [n], |A|  \leq 2s } \frac{\gamma_{min}(\Psi_{A}^t\Psi_{A})}{m},$$
and
 $$O_{s}(\Psi) = \sup_{A, B\subset [n], A \cap B = \emptyset , |A|\leq s, |B| \leq s}  \|\Psi_{A}^t\Psi_{B}/m\|.$$
We obtain that,
\begin{align}
\langle \Psi_{\mathcal{A}} z_{\mathcal{A}}, \Psi_{\mathcal{I}} z_{\mathcal{I}} \rangle/m & =
\langle \Psi_{\mathcal{A}} z_{\mathcal{A}}, \sum_{t\geq 1}\Psi_{\mathcal{I}_t} z_{\mathcal{I}_t} \rangle /m \nonumber\\
& \leq  \|z_{\mathcal{A}}\|_2 \sum_{t\geq 1} \|\Psi^t_{\mathcal{A}} \Psi_{\mathcal{I}_t}\|_2\|z_{\mathcal{I}_t}\|_2\nonumber\\
&\leq O_{s}(\Psi) \|z_{\mathcal{A}}\|_2 \sum_{t\geq 1} \|z_{\mathcal{I}_t}\|_2 \nonumber\\
&\leq 4 O_{s}(\Psi) \|z_{\mathcal{A}}\|^2,\label{oc}
\end{align}
where the first inequality uses Cauchy-Schwartz inequality, the second inequality follows from the definition of  $O_{s}(\Psi),$ and the third  is due to  \eqref{shift}.
Then, $\forall z \in \mathcal{C}_{\mathcal{A^*}}, z \neq 0,$
we have
\begin{align}
&\|\Psi z \|^2_2/(m \|z\|^2) \geq\|\Psi z \|^2_2/(17 m \|z_{\mathcal{A}}\|^2) \nonumber\\
&= (\|\Psi_{\mathcal{A}} z_{\mathcal{A}}\|^2 + \|\Psi_{\mathcal{I}} z_{\mathcal{I}}\|^2 + 2\langle\Psi_{\mathcal{A}} z_{\mathcal{A}}, \Psi_{\mathcal{I}} z_{\mathcal{I}} \rangle)/(17 m \|z_{\mathcal{A}}\|^2)\nonumber\\
& \geq (\|\Psi_{\mathcal{A}} z_{\mathcal{A}}\|^2  -  8O_{s}(\Psi)  \|z_{\mathcal{A}}\|^2)/(17 m \|z_{\mathcal{A}}\|^2)\nonumber\\
&\geq ( C_{2s}(\Psi)- 8O_{s}(\Psi))/17,\label{lowcs}
\end{align}
where the first inequality uses \eqref{sw}, the second inequality follows from \eqref{oc}, and the last holds due to the definition of
$C_{2s}(\Psi).$
 It follows from \eqref{lowcs} that,  to   complete the proof of this lemma it suffices to derive a lower  bound on $ C_{2s}(\Psi)$ and an upper bound on  $O_{s}(\Psi)$  with high probability, respectively.
Given $A\subset [n], |A|  \leq 2s  $, we define the event $E_{A} = \{\sqrt{\frac{\gamma_{min}(\Psi_{A}^t\Psi_{A})}{m}} > \sqrt{\gamma_{min}(\Sigma)}(1-C_{2} \sqrt{\frac{2s}{m}}-\frac{t}{\sqrt{m}})\}$. Then,
\begin{align}
\mathbb{P}[C_{2s}(\Psi) > \gamma_{min}(\Sigma)(1-C_{2} &\sqrt{\frac{2s}{m}}-\frac{t}{\sqrt{m}})^2] = \mathbb{P}[\bigcap_{A\in [n], |A|\leq 2s} E_{A}]\nonumber = \mathbb{P}[\bigcap_{A\in [n], |A|= \ell, 1\leq\ell\leq 2s} E_{A}]\nonumber\\
& = 1 - \mathbb{P}[\bigcup_{A\in [n], |A|= \ell, 1\leq\ell\leq 2s} \overline{E_{A}}]\nonumber\\
&\geq 1- \sum_{\ell = 1}^{2s}\sum_{A\subset[n], |A| =  \ell} (1-\mathbb{P}[E_{A}])\nonumber\\
&\geq 1- \sum_{\ell = 1}^{2s}\sum_{A\subset [n], |A|\leq \ell}  2\exp(-C_1 t^2)\nonumber\\
&= 1- \sum_{\ell = 1}^{2s}\binom{n}{\ell}2 \exp(-C_1 t^2 )\nonumber\\
&\geq 1- 2(\frac{en}{2s})^{2s} \exp(-C_1 t^2 ),\nonumber
\end{align}
where the first inequality follows from the union bound, the second inequality  follows from \eqref{norm1} by replacing $\Psi$ with $\Psi_{A}$, and  the third inequality  holds since $\sum_{\ell = 1}^{2s}\binom{n}{\ell}\leq (\frac{n}{2s})^{2s}\sum_{\ell =0}^{2s}\binom{n}{\ell}(\frac{2s}{n})^{\ell}\le (\frac{n}{2s})^{2s}(1+\frac{2s}{n})^{n} \leq (\frac{e n}{2s})^{2s}$.
Then, we derive with probability at least $1- 2(\frac{en}{2s})^{2s} \exp(-C_1 t^2 ),$
\begin{equation}\label{lbc}
C_{2s}(\Psi) > \gamma_{min}(\Sigma)(1-C_{2} \sqrt{\frac{2s}{m}}-\frac{t}{\sqrt{m}})^2.
\end{equation}
Given $A\subset [n],  B \subset [n], |A|  \leq s, |B|\leq s, A\cap B = \emptyset $, we define the event
$E_{A, B} = \{\|\Psi_{A}^t\Psi_{B}/m\| > \gamma_{max}(\Sigma)(C_{2} \sqrt{\frac{2s}{m}}+\frac{t}{\sqrt{m}})\}$.
Denote $C = A \cup B,$  $\Phi_C = \Psi_C \Sigma_{CC}^{-\frac{1}{2}}$, $G_{C} = \Phi_{C}^t\Phi_{C}/m - \textbf{I}_{2s}$.  Then each row of $\Phi_C $ is multivariate normal random vector that is sampled from $\mathcal{N}(\textbf{0}, \textbf{I}_{2s})$. It follows from  \eqref{norm2} with $\Psi$ and $\Sigma$ replaced by $\Phi_C$ and $\textbf{I}_{2s}$, respectively,  that
$$\mathbb{P}[ \|G_{C}\|\geq C_2 \sqrt{\frac{2s}{m}}+\frac{t}{\sqrt{m}} ] \leq 2\exp{(-C_1 t^2)}.$$
Observing $\Sigma_{CC}^{-\frac{1}{2}}\Psi_{A}^t\Psi_{B} \Sigma_{CC}^{-\frac{1}{2}}/m$ is a sub-matrix of $G_{C}$,
we deduce,
$$\mathbb{P}[ E_{A,B}] \leq 2\exp{(-C_1 t^2)}.$$
Then, similarly to the proof of \eqref{lbc}, we have
\begin{align}
\mathbb{P}[O_{s}(\Psi) > \gamma_{max}(\Sigma)(C_{2} \sqrt{\frac{2s}{m}}+\frac{t}{\sqrt{m}})] &= \mathbb{P}[\bigcup_{A, B\subset [n], A \cap B = \emptyset, |A|\leq s, |B| \leq s} E_{A,B}]\nonumber\\
& = \mathbb{P}[\bigcup_{A\subset [n], |A| = \ell, B\subset [n], |B| = \tilde{\ell}, A \cap B = \emptyset,  1\leq\ell\leq s, 1\leq\tilde{\ell}\leq s } E_{A,B}]\nonumber\\
& \leq \sum_{\ell = 1,\tilde{\ell} = 1}^{s}\sum_{A\in [n], |A|\leq \ell, B\in [n], |B|\leq \tilde{\ell},  A \cap B = \emptyset}  2 \exp(-C_1 t^2)\nonumber\\
&\leq (\sum_{\ell = 1}^{s}\binom{n}{\ell})^2 2\exp(-C_1 t^2 )\nonumber\\
&\leq  2(\frac{en}{s})^{2s} \exp(-C_1 t^2 ),\nonumber
\end{align}
which implies  with probability at least $1-2(\frac{en}{s})^{2s} \exp(-C_1 t^2 )$,
\begin{equation}\label{ubo}
O_{s}(\Psi) \leq  \gamma_{max}(\Sigma)(C_{2} \sqrt{\frac{2s}{m}}+\frac{t}{\sqrt{m}}).
\end{equation}
Combining \eqref{lbc} and \eqref{ubo} and setting $t = \sqrt{\frac{4s}{C_1} \log\frac{en}{s}}$, we obtain that  with probability at least
$1- 4/(\frac{en}{s})^{2s} \geq 1-  4/n^2$
$$ C_{2s}(\Psi) - 8O_{s}(\Psi) \geq \gamma_{min}(\Sigma)f\big(\sqrt{\frac{2s}{m}}+\sqrt{\frac{4s}{mC_1} \log\frac{en}{s}}\big),$$
where the unitary function $f(z) = z^2-(8\kappa(\Sigma)+2)z+1$.
It follows from  the assumption $s\leq \exp^{(1-\frac{C_1}{2})}n$ that $\sqrt{\frac{2s}{m}}\leq \sqrt{\frac{4s}{mC_1} \log\frac{en}{s}}$. Then some basic algebra shows that $f\big(\sqrt{\frac{2s}{m}}+\sqrt{\frac{4s}{mC_1} \log\frac{en}{s}}\big)\geq f(\frac{1}{8\kappa(\Sigma)+2}) = \frac{1}{4(4\kappa(\Sigma)+1)^2}$ as long as $m \geq \frac{64(4\kappa(\Sigma)+1)^2}{C_1}s\log\frac{en}{s}$.
The proof of Lemma \ref{rel} is completed.
\end{proof}

Now we are in  the place of combining  the above pieces together  to finish the  proof of Theorem $\ref{errsub}$. Recall $R = x_{\ell_{1}} - \widetilde{x}^*$.
It follows from Lemma \ref{slev} that   $R\in  \mathcal{C}_{\mathcal{A}^*}$ and \eqref{basiceq} holds by conditioning on $\mathcal{E}$,
i.e.,
 $$\frac{1}{m}\|\Psi R\|_2^2 + \lambda \|R_{\mathcal{I}^*}\|_1 \leq 3\lambda \|R_{\mathcal{A}^*}\|_1,$$
 which together with Lemma \ref{rel} implies that,   with probability at least $1-4/n^2$,
$$\frac{\gamma_{min}(\Sigma)}{68(4\kappa(\Sigma)+1)^2} \|R\|_2^2 \leq 3\lambda \|R_{\mathcal{A}^*}\|_1\leq  3\lambda \sqrt{s}\|R_{\mathcal{A}^*}\|_2,$$
i.e., $$ |c|\|x_{\ell_{1}}/c - {x}^*\|_2=\|x_{\ell_{1}} - \widetilde{x}^*\|_2 \leq \frac{204(4\kappa(\Sigma)+1)^2}{\gamma_{min}(\Sigma)}\lambda \sqrt{s}.$$
The proof of Theorem \ref{errsub} is completed by dividing $|c|$ on both side and using Lemma \ref{noiselinf}, which guaranties that $\eqref{basiceq}$ holds with    $\lambda =  \frac{4(1+|c|C_3)}{\sqrt{C_1}}\sqrt{\frac{\log n}{m}}$ with probability greater than  $1-2/n^3  -2/n^2$.
\end{proof}

%

\section{Proof of the equivalency between the PDAS and   \eqref{newton1} - \eqref{newton2}}
\label{app:shownewton}
\begin{proof}
  Partition $Z^k$,   $D^k$ and  $F(Z^k)$ according to  $\mathcal{A}_{k}$ and $\mathcal{I}_{k}$ such that
  \begin{equation*}
       Z^{k}=\left(
         \begin{array}{c}
         x_{\mathcal{A}_{k}} \\
           x_{\mathcal{I}_{k}} \\
           d_{\mathcal{A}_{k}} \\
          d_{\mathcal{I}_{k}} \\
         \end{array}
          \right),
          \end{equation*}
       \begin{equation}\label{FF2}
       D^{k}=\left(
         \begin{array}{c}
         D^{x}_{\mathcal{A}_{k}} \\
          D^{x}_{\mathcal{I}_{k}} \\
           D^{d}_{\mathcal{A}_{k}} \\
          D^{d}_{\mathcal{I}_{k}} \\
         \end{array}
          \right),
          \end{equation}
       \begin{equation}\label{FF3}
F(Z^{k})=
\left[\begin{array}{c}
 - d_{\mathcal{A}_{k}}^{k} + \lambda sign(x_{\mathcal{A}}^k+ d_{\mathcal{A}}^k)  \\
 x_{\mathcal{I}_{k}}^{k}  \\
  \Psi^{t}_{\mathcal{A}_{k}} \Psi_{\mathcal{A}_{k}} x_{\mathcal{A}_{k}}^{k} +  \Psi^t_{\mathcal{A}_{k}} \Psi_{\mathcal{I}_{k}} x_{\mathcal{I}_{k}}^{k}+ m d_{\mathcal{A}_{k}}^{k} - \Psi^t_{\mathcal{A}_{k}}y     \\
\Psi^{t}_{\mathcal{I}_{k}} \Psi_{\mathcal{A}_{k}} x_{\mathcal{A}_{k}}^{k} +  \Psi^{t}_{\mathcal{I}_{k}} \Psi_{\mathcal{I}_{k}} x_{\mathcal{I}_{k}}^{k}+ m d_{\mathcal{I}_{k}}^{k} - \Psi^t_{\mathcal{I}_{k}}y
\end{array}\right].
\end{equation}
 Substituting \eqref{FF2} - \eqref{FF3} and \eqref{Jacobi} into \eqref{newton1}, we have
\begin{align}
  -(d_{\mathcal{A}_{k}}^{k} + D^{d}_{\mathcal{A}_{k}})    &= -\lambda sign(x_{\mathcal{A}}^k+ d_{\mathcal{A}}^k),\label{eqv1}\\
   x^{k}_{\mathcal{I}_{k}} + D^{x}_{\mathcal{I}_{k}}  &= \textbf{0},\label{eqv2}\\
\Psi^{t}_{\mathcal{A}_{k}} \Psi_{\mathcal{A}_{k}} (x_{\mathcal{A}_{k}}^{k} + D^{x}_{\mathcal{A}_{k}})&= \Psi^t_{\mathcal{A}_{k}}y  - m( d_{\mathcal{A}_{k}}^{k} + D^{d}_{\mathcal{A}_{k}})- \Psi^t_{\mathcal{A}_{k}} \Psi_{\mathcal{I}_{k}} ( x^{k}_{\mathcal{I}_{k}} + D^{x}_{\mathcal{I}_{k}}),  \label{eqv3}\\
m(d_{\mathcal{I}_{k}}^{k} + D^{d}_{\mathcal{I}_{k}}) &= \Psi^t_{\mathcal{I}_{k}}y - \Psi^{t}_{\mathcal{I}_{k}} \Psi_{\mathcal{A}_{k}}(x_{\mathcal{A}_{k}}^{k} + D^{x}_{\mathcal{A}_{k}}) - \Psi^t_{\mathcal{A}_{k}} \Psi_{\mathcal{I}_{k}}(x_{\mathcal{I}_{k}}^{k} + D^{x}_{\mathcal{I}_{k}}) . \label{eqv4}
\end{align}
It follows from \eqref{newton2} that
 \begin{equation}\label{relation}
 \left( \begin{array}{c}
  x_{\mathcal{A}_{k}}^{k+1} \\
          x_{\mathcal{I}_{k}}^{k+1} \\
           d_{\mathcal{A}_{k}}^{k+1} \\
          d_{\mathcal{I}_{k}}^{k+1} \\
         \end{array}
       \right)
       = \left(
         \begin{array}{c}
           x_{\mathcal{I}_{k}}^{k}+D^{x}_{\mathcal{I}_{k}} \\
           d_{\mathcal{A}_{k}}^{k} +D^{d}_{\mathcal{A}_{k}}\\
           x_{\mathcal{A}_{k}}^{k}+D^{x}_{\mathcal{A}_{k}} \\
          d_{\mathcal{I}_{k}}^{k}+ D^{d}_{\mathcal{I}_{k}} \\
         \end{array}
       \right).
 \end{equation}
Substituting  \eqref{relation} into \eqref{eqv1} - \eqref{eqv4}, we get the iteration  procedure of PDAS in Algorithm \ref{alg:genew}.
This completes the  proof.
\end{proof}

\bibliographystyle{siam}
\bibliography{onebitref}
\end{document}